\documentclass[11pt,leqno]{article}
\usepackage{amsmath, amsfonts,color}
\usepackage[latin1]{inputenc}
\usepackage[T1]{fontenc}
\usepackage[english]{babel}
\usepackage{lmodern}
\usepackage{ dsfont }
\usepackage{mathrsfs}
\usepackage{amsmath}
\usepackage{amsthm}
\usepackage{amssymb}
\usepackage{mathabx}
\usepackage{latexsym}
\usepackage{graphicx}
\usepackage{color}
\usepackage{enumitem}

\newcommand{\R}{{\mathbb R}}
\newcommand{\N}{{\mathbb N}}
\newcommand{\re}{{\mathbb R}}

\newcommand{\one}{{\mathds{1}}}

\newcommand{\cP}{{\mathcal P}}

\let\ds=\displaystyle

\def\R{{\mathbb R}}
\def\N{{\mathbb  N}}
\def\N{{\mathbb N}}
\def\ds{\displaystyle}

\numberwithin{equation}{section}

 \newtheorem{theorem}{Theorem}[section]
 \newtheorem{lemma}{Lemma}[section]
 \newtheorem{proposition}{Proposition}[section]
 \newtheorem{definition}{Definition}[section]
 \newtheorem{corollary}{Corollary}[section]
 \newtheorem{remark}{Remark}[section]
 
 \newtheorem{assumption}{Assumption}[section]
 
\setlist[itemize,1]{leftmargin=\dimexpr 26pt-.3in}
\setlist[enumerate,1]{leftmargin=\dimexpr 26pt-.3in}
\setlist[description,1]{leftmargin=\dimexpr 26pt-.3in}
\setlist[itemize,2]{leftmargin=\dimexpr 26pt-.1in}
\setlist[enumerate,2]{leftmargin=\dimexpr 26pt-.1in}
\setlist[description,2]{leftmargin=\dimexpr 26pt-.1in}

\begin{document}

\title{Deterministic mean field games with control on the acceleration and state constraints}

\author{{
Yves Achdou \thanks { Universit{\'e} de Paris and  Sorbonne Universit{\'e}, CNRS, Laboratoire Jacques-Louis Lions, (LJLL), F-75006 Paris, France, achdou@ljll-univ-paris-diderot.fr},
 Paola Mannucci\thanks{Dipartimento di Matematica ``Tullio Levi-Civita'', Universit{\`a} di Padova, mannucci@math.unipd.it}}, \\
{Claudio Marchi \thanks{Dipartimento di Ingegneria dell'Informazione, Universit{\`a} di Padova, claudio.marchi@unipd.it},
  Nicoletta Tchou\thanks{Univ Rennes, CNRS, IRMAR - UMR 6625, F-35000 Rennes, France, nicoletta.tchou@univ-rennes1.fr}}}

\maketitle
\begin{abstract}
  We consider deterministic mean field games in which the agents control their acceleration and are constrained to remain in a domain of $\R^n$. We study relaxed equilibria in the Lagrangian setting; they  
are described by a probability measure on trajectories. The main results of the paper concern the existence of relaxed equilibria under suitable assumptions.  The fact that the optimal trajectories of the 
related optimal control problem solved by the agents do not form a compact set brings a difficulty in the proof of existence. The proof also requires closed graph properties of the  map which associates to initial conditions the set of  optimal trajectories. 
\end{abstract}

\section{Introduction}
\label{sec:introduction}

 The theory of  mean field games (MFGs for short) is more and more investigated since the pioneering works
\cite{MR2269875,MR2271747,MR2295621} of  Lasry and Lions: it aims at studying the  asymptotic behaviour of   differential games (Nash equilibria) as the number  of agents tends to infinity. The dynamics of the agents can be either stochastic or deterministic. Concerning the latter case, we refer to \cite{cardaliaguet2010} for a detailed study of  deterministic MFGs in which the interactions between the agents are modeled by a nonlocal regularizing operator acting on the distribution of  the states of the agents.
They are described by a system of PDEs coupling a continuity  equation for the density of the distribution of states (forward in time) and a  Hamilton-Jacobi  (HJ) equation for the optimal value of a representative agent (backward in time).  If the interaction cost depends  locally on the density of the distribution (hence is not regularizing), then, in the deterministic case, the available theory  mostly deals with so-called variational MFGs, see \cite{cardaliaguet2014second}.

 The major part of the literature on deterministic mean field games addresses situations
 when the dynamics of a given agent is strongly controllable: for example, in crowd motion models,
 this happens if the control of a given agent is its velocity. 
 Under the strong controllability assumption, it is possible to study realistic models
 in which the agents are constrained to remain in a given region $K$ of the space state,
 i.e. state constrained deterministic MFGs. An important  difficulty in state constrained 
deterministic MFGs is that nothing prevents  the agents from  concentrating on the boundary 
 $\partial K$ of the state space; let us call $m(t)$ the distribution of states at time $t$.
Even if $m(0)$ is absolutely continuous, there may exist some 
 $t>0$, such that $m(t)$ has a singular part supported on $\partial K$ and the absolute continuous part
of $m(t)$ with respect to Lebesgue measure blows up near $\partial K$.
  This was first observed in some applications of MFGs to macroeconomics, 
see \cite{MR3268061,achdou2017income}. From the theoretical viewpoint, 
the main issue is that, as we have already said, the distribution of states 
is generally not absolutely continuous with respect to Lebesgue measure;
 this makes it difficult to 
characterize the state distribution by means of partial differential equations.
 These theoretical difficulties have been addressed in \cite{MR3888967}: 
following ideas contained in \cite{BenamouBrenier2000monge,BenamouCarlier2015ALG2,MR3556062},
 the authors of  \cite{MR3888967} introduce a weak or relaxed notion of equilibrium,
 which is defined in a  Lagrangian setting rather than with PDEs. 
Because there may be several optimal trajectories starting from a given point
 in the state space, the solutions of the  relaxed MFG are  probability measures
 defined on a set  of  admissible trajectories. 
 Once the existence of a relaxed equilibrium is ensured,
 it is then possible to investigate the regularity of solutions
 and give a meaning to the system of PDEs and the related boundary conditions:
 this was done in  \cite{cannarsa:hal-01964755}.

On the other hand, if the agents control their acceleration rather than their velocity, the strong controllability property is lost.
 In \cite{MR4102464}, we have studied deterministic mean field games in the whole space  $\R^n$  with finite time horizon in which the dynamics of a generic agent is controlled by the acceleration, see also \cite{MR4132067}.
 The state variable is the pair $(x,v)\in \R ^n\times \R^n$  where $x$ and $v$ respectively stand  for the position and  the velocity. The system of PDEs describing the MFG is then \begin{equation}
\label{eq:MFGA}
\left\{
\begin{array}{rll}
(i)&-\partial_t u-v\cdot D_xu+H(x,v,D_vu)-F[m(t)](x,v)=0&\qquad \textrm{in }\re^{2n}\times (0,T)\\
(ii)& \partial_t m +v\cdot D_xm-{\rm div}_v(D_{p_v}H(x,v,D_vu)m)=0&\qquad \textrm{in }\re^{2n}\times (0,T)\\
(iii)& m(x,v,0)=m_0(x,v), u(x,v,T)=G[m(T)](x,v)\,,&\qquad \textrm{on }\re^{2n}
\end{array}\right.
\end{equation}
where  $u=u(x,v,t)$, $m=m(x,v,t)$, $(x,v)\in\R^{2n}$, $t\in(0,T)$ and the Hamiltonian
$ (x,v,p_x,p_v)\mapsto -v\cdot p_x+H(x,v,p_v)$  is neither  strictly convex  nor  coercive
with respect to $p=(p_x, p_v)$. Hence the available results 
on the regularity of the value function $u$ of the associated optimal control problem 
(\cite{MR2041617,cardaliaguet2010})  and on the
existence of a solution of the MFG system (\cite{cardaliaguet2010}) cannot be applied.
In \cite{MR4102464},  the existence of a weak solution of the MFG system is proved via a vanishing viscosity method;  the distribution of states is characterized as the image of the initial distribution by the flow associated with the optimal control. 

In traffic theory and also in economics, the models may  require that the position of the agents
belong to a given compact subset  $\overline \Omega$ of $\R^n$, and state constrained mean field games with control on the acceleration must be considered. 
In the present paper, we wish to investigate some examples of such mean field games 
 and address  the first step of the program followed by the authors of \cite{MR3888967} in the strongly controllable case:  we wish to prove the existence of a relaxed mean field equilibrium in the Lagrangian setting under suitable assumptions. 

\subsection{Our program}\label{sec:our-program}
Most of the paper is devoted to the case when the running cost depends separately on the acceleration and on the other  variables, and is quadratic with respect to the acceleration. We will assume that the acceleration can be chosen in the whole space. Thus,  admissible trajectories are  pairs of functions $(\xi, \eta)$,  $\xi\in C^1([0,T]; \overline \Omega)$, $\eta\in W^{1,2}([0,T];\R^n)$ and $\xi'=\eta$. 
An example of state constrained mean field games in which the  acceleration takes its values in a 
compact of $\R^n$, (the optimal value  may therefore take the value $+\infty$ in the interior of the $x$-domain), will be studied in a forthcoming work. 

\medskip

In view of the  applications to traffic models, we will  deal with the cases  when 
\begin{enumerate}
\item $\Omega$ is a bounded domain of $\R^n$ with a smooth boundary
\item  $n=1$ and $\Omega$ is a bounded straight line segment
\item $\Omega$ is a  bounded polygonal domain of $\R^2$ (that we suppose convex for simplicity).
\end{enumerate}
 In the one-dimensional case, the simplicity of the geometry will allow us to obtain accurate information on the optimal trajectories, and in turn to get a more general existence result for the mean field game, yet under an additional assumption on the running cost.  On the contrary, when dealing with the polygonal domain, we will  make special assumptions in order to obtain an existence result.

\bigskip

Recall that the admissible states are the pairs $(x,v)\in \overline\Omega\times \R^n$, where $\Omega$ is a bounded domain of $\R^n$. At first glance, we see that  some restrictions will have to be
 imposed on the initial distribution of states: 
  indeed, for $x\in \partial \Omega$ and  $v$ pointing outward $\Omega$ at $x$,
 there is no admissible trajectory taking the value $(x,v)$ at $t=0$; 
hence the optimal value $u(x,v,0)$ takes the value $+\infty$;
 the definition of the mean field equilibrium would then be unclear 
if the probability that the initial state takes such values $(x,v)$ was not zero.
 
As in \cite{MR3888967}, the aim is to prove the existence of  relaxed MFG equilibria
which are described by  probability measures defined on a  set  of admissible trajectories.
The proof  involves Kakutani's fixed point theorem,  see \cite{MR46638},
 applied to a multivalued map defined on a suitable  convex and compact set of 
probability measures on a suitable set of admissible trajectories (itself endowed with the $ C^1([0,T]; \R^n )\times C^0([0,T]; \R^n )$-topology).  Difficulties in applying  Kakutani's fixed point theorem will arise from the fact that  all the optimal trajectories do not form a compact subset of $ C^1([0,T]; \R^n )\times C^0([0,T]; \R^n )$ (due to the lack of strong controllability). This explains why we shall need additional assumptions, either on the support of the initial distribution of states, or, in some cases, on the running cost.


\paragraph{Assumptions on the support of the initial distribution of states}
Note that if a set of trajectories is a compact metric space, then  probability measures 
on this set form a compact set, as required by Kakutani's theorem. Therefore,
 a natural strategy is to identify a  compact set of trajectories which contains
 the optimal trajectories whose initial value belongs to the support
 of the initial distribution of states.
In such a strategy, we therefore need to identify a modulus of continuity common to all the velocity laws
 of the optimal
 trajectories; since the running cost is quadratic in the acceleration, 
the more natural idea is to look for a uniform  bound on the $W^{1,2}$ norms of the velocity 
laws of the optimal trajectories. But, due the lack of strong controllability, if  $x$ and 
 $v$ respectively belong to  $\partial \Omega$ and  to the boundary of the tangent cone 
to $\overline \Omega$ at $x$ (the optimal value $u(x,v,0)$ is finite), there exist sequences 
 $(x_i,v_i)_{i\in \N}$ tending to $(x, v)$ such that the optimal value  $u(x_i,v_i,0)$ blows up when $i\to \infty$; in other words, the cost of preventing the trajectories with initial value $(x_i,v_i)$ from exiting the domain  tends to $+\infty$ as $i\to \infty$.
 Hence, to get uniform bounds on the $W^{1,2}$ norms of the velocity law,
 the support of the initial distribution of states must not contain such
 sequences $(x_i, v_i)$. Sufficient conditions on the support of the initial distribution will
 be given.

Furthermore,  Kakutani's fixed point theorem  requires a closed graph property for the multivalued map which maps a given point $(x,v)$ to the set of optimal trajectories starting from $(x,v)$.  An important part of our work is therefore  devoted to proving a closed graph property for the latter map.
Note that this issue has its own interest in optimal control theory, independently from mean field games.

\paragraph{Assumptions on the running cost}
We will see that if $n=1$, we will able to get rid of the above-mentioned restrictions on the support of the initial distribution of states, if an additional assumption is made on the running cost, namely that it does not favor the trajectories that exit the domain. The existence of equilibria is then proved by  approximating the initial distribution $m_0$  by a sequence $m_{0,k}$ for which Kakutani's theorem can be applied, and by passing to the limit.  To pass to the limit, accurate information on the optimal trajectories are needed. We managed to obtain them  for $n=1$  only.

\bigskip

\subsection{Organization of the paper}\label{sec:organization-paper}
The paper is organized as follows: Section~\ref{sec:state-constr-optim-1} is devoted to state constrained optimal control problems in a bounded domain of $\R^n$ with a smooth boundary, and in particular to the closed graph properties of the above mentioned multivalued map. Although this issue seems to be important in several applications, we were not able to find any relevant result in the available literature. Then, Section \ref{sec:MFG-contr-accel} deals with an existence result for a related mean field equilibrium in the Lagrangian setting, under sufficient conditions on the support of the initial distribution of states. A variant with a non quadratic cost will be investigated as well.   In Section~\ref{sec:more-accur-results}, we address the case when the dynamics take place in a bounded straight line segment, ($n=1$): under a natural additional assumption on the running cost, we are able to  prove the existence of mean field equilibria without any restriction on the  initial distribution of states; the proof requires a quite  careful study of the optimal trajectories. Finally, in Section~\ref{sec:contr-probl-polyh}, we discuss the case of bounded and convex polygonal domain of $\R^2$ and put the stress on the closed graph result which requires a special care near the corners.

\section{State constrained optimal control problems in a  domain of $\R^n$}
\label{sec:state-constr-optim-1}
\subsection{Setting and notation}
\label{sec:setting-notation-8}
Let  $\Omega$ be a  bounded domain of $\R^n$  with a boundary $\partial \Omega$ of class $C^2$.
 For $x\in \partial \Omega$, let $n(x)$ be the unitary vector normal to $\partial \Omega$  pointing outward 
$\Omega$.
We will use the signed distance to  $\partial \Omega$, $d: \R^n\to \R$,
\begin{displaymath}
 d(x)=\left\{
    \begin{array}[c]{rcl}
      \min_{y\in \partial \Omega} |x-y|,\quad&\hbox{if}&\quad x\notin \Omega,\\
      -\min_{y\in \partial \Omega} |x-y|,\quad&\hbox{if}&\quad x\in \Omega.
    \end{array}
\right.
\end{displaymath}
Since $\partial \Omega$ is $C^2$, the function $d$ is $C^2$ near  $\partial \Omega$. In particular, for all $x\in \partial \Omega$, $\nabla d(x)= n(x)$.

Given a time horizon $T$ and a pair  $(x,v)\in \overline \Omega \times \R^n$, we are interested in optimal control problems for which the dynamics is of the form:
\begin{equation}\label{eq:1}
\left\{
\begin{array}{rcll}
 \xi'(s)&=&\eta(s),\quad   &s\in (0,T),\\
 \eta'(s)&=&\alpha(s),\quad   &s\in (0,T),\\
\xi(0)&=&x, &\\ \eta(0)&=&v.& 
\end{array}\right.
\end{equation}
The state variable is the pair $(\xi, \eta)$ and the state space is  $\Xi=\overline \Omega \times \R^n$. The optimal control problem consists of minimizing the cost
\begin{equation}
\label{eq:2}
J(\xi,\eta, \alpha)
=\int_0^T \left(\ell(\xi(s),\eta(s) ,s) +\frac 1 2   |\alpha|^2(s)
 \right) ds+g(\xi(T), \eta(T)),
\end{equation}
on the dynamics given by \eqref{eq:1} and staying in $\Xi$.
\begin{assumption}
    \label{sec:setting-notations-0}
    Here, $\ell: \Xi\times [0,T]\to \R$ is a  continuous function, bounded from below.
    The terminal cost $g: \Xi\to \R$ is also assumed to be  continuous and bounded from below.
     Set 
     \begin{equation}
       \label{eq:3}
       M= \|g_-\|_{L^\infty (\Xi) }  +   \|\ell_-\|_{L^\infty( \Xi\times [0,T]) } .
     \end{equation}

\end{assumption}
It is convenient to define the set of admissible trajectories as follows:
\begin{equation}
\label{eq:4}
\Gamma=\left\{
  \begin{array}[c]{ll}
(\xi, \eta)\in C^1([0,T];\R^n) \times AC([0,T];\R^n) \;:\; &\left|
\begin{array}[c]{ll}
  \xi'(s)=\eta(s), \; &\forall s\in [0,T] 
  \\   (\xi(s),\eta(s))\in \Xi,  &\forall s\in [0,T] 
\end{array}\right.  \end{array}
\right\}.
\end{equation}
For any $(x,v)\in \Xi$, set
\begin{equation}
\label{eq:5}
\Gamma[x,v]=\{ (\xi, \eta) \in \Gamma:\,\xi(0)=x,\, \eta(0)=v\}.
\end{equation}
Then,  $\Gamma^{\rm{opt}}[x,v]$ is the set of all $ (\xi, \eta)  \in  \Gamma[x,v]$ such that $\eta \in W^{1,2}(0,T, \R^n) $ and
$(\xi, \eta, \eta')$ achieves the minimum of $J$ in $\Gamma[x,v]$.
\\
 Note that $\Gamma[x,v]=\emptyset$ if $x\in \partial \Omega$ and $v$ points outward $\Omega$.
 This is the reason why we introduce $\Xi^{\rm{ad}}$ as follows:
  \begin{equation}
    \label{eq:6}
\Xi^{\rm{ad}}=\{  (x,v):  x\in \overline \Omega, \;  v\cdot n(x)\le 0 \hbox{ if } x\in \partial \Omega
 \}\subset \Xi.
  \end{equation}

  \begin{lemma}
    \label{sec:contr-probl-conv}
For all $(x,v)\in \Xi^{\rm{ad}}$, the optimal value 
\begin{equation}
  \label{eq:7}
u(x,v)=\inf_{(\xi,\eta)\in \Gamma[x,v]}J(\xi,\eta, \eta')
\end{equation}
 is finite.
The function $u$ is lower semi-continuous on $ \Xi^{\rm{ad}}$.
  \end{lemma}
  \begin{proof}
Let us consider $(x,v) \in \Xi^{\rm{ad}}$. We make out two cases:
\begin{enumerate}
\item $x\in \Omega$ or $x\in \partial \Omega$ and $v\cdot n(x)<0$: in this case, for $\overline t$ small enough,  the  trajectory $(\xi,\eta)$ defined by 
  \begin{eqnarray*}
\left\{
         \begin{array}[c]{rcllrclrcl}
           \eta(s)&=& (1- \frac s {\overline t})v\quad&\hbox{and}\quad  &\xi(s)=& x  + (s- \frac {s^2} {2\overline t})v,\quad  &\hbox{if }\quad  &  0\le s\le \overline t,\\
 \eta(s)&=& 0 \quad &\hbox{and} \quad &\xi(s)=& x+ \frac {\overline t } 2 v   ,\quad &\hbox{if }\quad  &   \overline t\le s \le T,
\end{array}
\right.
\end{eqnarray*}
 is admissible and $J(\xi,\eta, \eta')$ is finite.
\item $x\in \partial \Omega$ and $v\cdot n(x)=0$.  We make a simple observation that will also be used in the proof of Lemma \ref{sec:clos-graph-prop-12} below: 
for all $x\in \partial\Omega$, there exists an open neighborhood $V_x$ of $x$ in $\R^n$, a positive number $R_x$ and a $C^2$-diffeomorphism $\Phi_x$ from $V_x$ onto $B(0, R_x)$ such that for all $y\in V_x$,
the $n^{\rm{th}}$ coordinate of $\Phi_x(y)$ is $d(y)$, i.e. $\Phi_{x,n}(y)= d(y)$. Hence, $ \Phi_x|_{V_x\cap \Omega}$ is a $C^2$-diffeomorphism  from $V_x\cap \Omega$ onto $ B_-(0, R_x)=B(0, R_x)\cap \{x_n<0\}$, and 
$ \Phi_x|_{V_x\cap \partial\Omega}$ is a $C^2$-diffeomorphism  from $V_x\cap \partial \Omega$ onto 
$B(0, R_x)\cap \{x_n=0\}$. Let us also call $\Psi_x$ the inverse of $\Phi_x$, which is a $C^2$-diffeomorphism  from $B(0, R_x)$ onto $V_x$.  Note that 
\begin{equation}
  \label{eq:8}
\nabla d(y)=D\Phi_x^T (y) e_n, \quad \hbox{for all } y\in V_x,
\end{equation}
where $e_n$ is the $n^{\rm{th}}$ vector of the canonical basis. In particular, $n(x)=D\Phi_x^T (x) e_n$.

In the present case, let us set $\hat x= \Phi_x(x)$ and $\hat v= D\Phi_x(x) v$. It is easy to see that $\hat x_n=0$ and $\hat v_n=0$. Then, for $\overline t$ small enough,  the  trajectory $(\xi,\eta)$ defined by $\xi(s)=\Psi_x(\hat \xi(s))$, $\eta(s)= \frac{d \xi}{d s}(s)= D\Psi_x(\hat \xi(s))\hat \eta(s) $ for all $s\in [0,T]$, with  \begin{eqnarray*}
\left\{
         \begin{array}[c]{rcllrclrcl}
           \hat\eta(s)&=& (1- \frac s {\overline t})\hat v\quad & \hbox{and}\quad  &\hat\xi(s)=& \hat x  + (s- \frac {s^2} {2\overline t})\hat v,\quad  &\hbox{if }\quad  &  0\le s\le \overline t,\\
\hat \eta(s)&=& 0 \quad&\hbox{and}\quad &\hat \xi(s)=& \hat x+ \frac {\overline t } 2 \hat v   ,\quad &\hbox{if }\quad  &   \overline t\le s \le T,
\end{array}
\right.
\end{eqnarray*}

is admissible and $J(\xi,\eta, \eta')$ is finite.
\end{enumerate}
 The lower semi-continuity of $u$  on  $\Xi^{\rm{ad}}$ stems from standard arguments in  the calculus of variations.
  \end{proof}

   \subsection{Closed graph properties}
\label{sec:clos-graph-prop}
An important feature  of the optimal control problem described above is the closed graph property:

\begin{proposition}\label{sec:contr-probl-polyg-1}
  Consider a closed subset $\Theta$ of $\Xi^{\rm ad}$.
Assume that  for all sequences  $(x^i, v^i)_{i\in \N}$  such that for all $i\in \N$, $ (x^i, v^i)\in \Theta$ and $\lim_{i\to +\infty} (x^i, v^i)=(x, v)\in \Theta$, the following holds:
\\
if  $x\in  \partial \Omega$,  then 
\begin{equation}
  \label{eq:9}
\left((v^i \cdot \nabla d (x^i) )_+\right)^3 = o\left( \left |d(x^i)\right| \right) ;
\end{equation}
then  the graph of the multivalued map
\begin{displaymath}
  \begin{split}
    \Gamma^{\rm opt}: \; &\Theta \rightarrow \Gamma,
    \\
    &    (x,v) \mapsto \Gamma^{\rm opt}[x,v]      
  \end{split}
\end{displaymath}
is closed, which means: for  any sequence $(y^i, w^i)_{i\in \N}$ such that for all $i\in \N$,   $(y^i, w^i)\in  \Theta$ with $(y^i, w^i)\to (y,w)$ as $i\to \infty$, consider a sequence     $(\xi^i,\eta^i)_{i\in \N}$ such that for all $i\in \N$, $(\xi^i, \eta^i) \in \Gamma^{\rm{opt}}[y^i, w^i]$; if $(\xi^i,\eta^i)$ tends to $(\xi,\eta)$ uniformly,  then $(\xi,\eta)\in \Gamma^{\rm{opt}}[y, w]$.
\end{proposition}

\begin{remark}
  \label{sec:clos-graph-prop-3}
In Proposition \ref{sec:contr-probl-polyg-1}, the condition (\ref{eq:9}) is restrictive only for sequences  $(x^i, v^i)\in \Theta$ which tend to
  $(x,v)\in \Theta$ such that  $x\in \partial \Omega$ and $v$ is tangent to $\partial \Omega$ at $x$.
We will see that this assumption makes it possible to control the cost associated to the optimal trajectories starting from $ (x^i, v^i)$.
\end{remark}
\begin{remark}
  \label{sec:clos-graph-prop-11}
 In Section \ref{sec:optim-contr-probl} below, we will see that in dimension one ($\Omega$ is then a bounded straight line), and under stronger assumptions on the running cost,  the closed graph properties hold for  $\Theta=\Xi^{\rm{ad}}$.
\end{remark}

\begin{remark}
  \label{sec:clos-graph-prop-9}
In the  context of mean field games, see Section \ref{sec:MFG-contr-accel} below, the assumptions in Proposition \ref{sec:contr-probl-polyg-1} will yield  sufficient conditions on the support of the initial distribution for the existence of relaxed mean field equilibria.  
\end{remark}

The proof of Proposition \ref{sec:contr-probl-polyg-1} relies on Lemmas~\ref{sec:clos-graph-prop-12} and \ref{sec:clos-graph-prop-7} below.

\begin{lemma}
\label{sec:clos-graph-prop-12}
  Consider  $(x,v)\in\Xi^{\rm{ad}}$, $(\xi,\eta)\in \Gamma[x,v]$ such that $\eta\in W^{1,2} (0,T;\R^n)$ and a sequence $(x^i, v^i)_{i\in \N}$ such that for all $i\in \N$,   $(x^i, v^i)\in\Xi^{\rm{ad}}$ 
  and $\ds \lim_{i\to \infty} (x^i, v^i)= (x,v)$.
  \\
  Assume that one among the following  conditions is true:
  \begin{enumerate}
  \item $x\in \Omega$
  \item $x\in \partial \Omega$ and  $v\cdot n(x) < 0$  (hence $v^i\cdot \nabla d(x^i)< 0$ for $i$ large enough)
  \item $x\in \partial \Omega$, $v\cdot n(x) = 0$ and one among  the following properties is true:
    \begin{enumerate}
    \item  for $i$ large enough,  $v^i\cdot \nabla d(x^i)\le 0$ 
    \item  for $i$ large enough, $v^i\cdot \nabla d(x^i)> 0$  (hence $d(x^i)<0$) and
      \begin{equation}
        \label{eq:12}
   \lim_{i\to \infty}   \frac { (v^i \cdot \nabla d (x^i) )^3} { |d(x^i)| }=0 .
      \end{equation}
    \end{enumerate}
  \end{enumerate}
 Then  there exists a sequence  $(\xi^i,\eta^i)_{i\in \N}$ such that  $(\xi^i,\eta^i)\in \Gamma [x^i, v^i]$, $\eta^i\in W^{1,2} (0,T;\R^n)$, and  $(\xi^i,\eta^i)$ tends to $(\xi,\eta)$ 
  in $W^{2,2}(0,T;\R^n) \times  W^{1,2}(0,T;\R^n)$, hence uniformly in $[0,T]$.
\end{lemma}
Before proving Lemma~\ref{sec:clos-graph-prop-12},  let us define a family of third order polynomials with values in $\R^n$:
\begin{definition}\label{sec:clos-graph-prop-2}
  Given $t>0$ and  $x,v,y,w\in \R^n$, let $Q_{t,x,v,y,w}$ be the unique third order polynomial  with value in $\R^n$ such that 
  \begin{equation}
    \label{eq:13}
Q_{t,x,v,y,w}(0)=x,\quad Q'_{t,x,v,y,w}(0)=v,\quad Q_{t,x,v,y,w}(t)=y,\quad Q'_{t,x,v,y,w}(t)=w.
  \end{equation}
It is given by 
\begin{equation}\label{eq:14}
  Q_{t,x,v,y,w}(s)= x+vs +\left(3\frac {y-x-vt} {t^2} - \frac {w-v} t \right) s^2 +\left(-2\frac {y-x-vt} {t^3} + \frac {w-v} {t^2} \right) s^3.
\end{equation}
The first and second derivatives of $Q_{t,x,v,y,w}$ are deduced immediately from  (\ref{eq:14}).
\end{definition}

\begin{proof}[Proof of Lemma~\ref{sec:clos-graph-prop-12}]
We are going to see that each of the three conditions  mentioned in the statement  makes it possible to explicitly  construct 
families of admissible trajectories fulfilling all the desired properties (in particular  with a finite 
energy or cost). The more tricky situations will arise when $x\in \partial \Omega$ and $v^i\cdot \nabla d(x^i)>0$
for $i$ large enough, in which case the restrictive condition (\ref{eq:12}) will be needed.  Since the construction is different in each of the three cases mentioned in Lemma~\ref{sec:clos-graph-prop-12}, we discuss each case separately.

\medskip

  \begin{enumerate}
  \item If $x\in \Omega$, then there exists $\overline t\in (0,T]$ and
    $c>0$ such that $ d(\xi(s))<-c$ for all $s\in [0,\overline t]$.
    We construct the sequence $(\xi^i, \eta^i)_{i\in \N}$ as follows:
    \begin{displaymath}
      \xi^i(s)=\left\{
        \begin{array}[c]{rcl}
           \xi(s)+ Q_{\overline t, \delta x^i, \delta v^i, 0,0 }(s),\quad &\hbox{if }& 0\le s \le \overline t,\\
           \xi(s),\quad &\hbox{if }& \overline t\le s \le T,
        \end{array}
\right.
    \end{displaymath}
where $\delta x^i=x^i-x$ and $\delta v^i=v^i-v$,
 see Definition \ref{sec:clos-graph-prop-2} for the third order  polynomial 
$Q_{\overline t,\delta x^i, \delta v^i, 0,0 }$. 
 It is clear that for $i$ large enough, $\xi^i(s)\in \overline \Omega$ for all $s\in [0,T]$, hence  $(\xi^i, \eta^i)\in \Gamma[x^i,v^i]$ and  $\eta^i\in W^{1,2} (0,T;\R^n)$.  On the other hand, it can be easily checked that
    \begin{equation}
      \label{eq:17}
     \lim_{i\to +\infty } \int_0^T \left\vert  \frac  {d\eta^i}{dt}(s) -\frac {d\eta}{dt}(s)
      \right\vert^2 ds = 0.
    \end{equation}
 This achieves the proof in the first case.
\item $x\in \partial \Omega$ and  $v\cdot n(x)<0$, hence
 for   $i$ large enough,   $v^i\cdot \nabla d( x^i)<0$.  We can always assume that the latter property holds for all $i$.

\medskip

\paragraph{Notation}
We use the same geometric arguments as in the proof of Lemma~\ref{sec:contr-probl-conv}:
 for the neighborhood $V_x$ mentioned there, there exists
$\hat T$, $0<\hat T \le  T$ such that $\xi(s)\in V_x\cap \overline \Omega$ for all $s\in [0,\hat T]$.
Consider the local chart $\Phi_x$ introduced in the proof of Lemma~\ref{sec:contr-probl-conv} and call $\Psi_x$ the inverse map,
from $B(0, R_x)$ onto $V_x$. For $t\in [0, \hat T]$, let us set  $\hat \xi(t)= \Phi_x(\xi(t))$,
$\hat \eta (t)= \frac {d\hat \xi} {dt}(t)= D\Phi_x(\xi(t))  \frac {d \xi} {dt}(t) $,   $\hat x= \Phi_x(x)$ and 
$\hat v= \hat \eta (0)= D\Phi_x(x)v $. It is clear that $\hat x_n=0$ and that $\hat v_n <0$. We also set $\hat x^i=\Phi_x(x ^i)$ and $\hat v^i=D\Phi_x(x ^i)  v^i$.

\medskip

Since $\eta\in W^{1,2}(0,T)$, there exists $\overline{t}\in (0,\hat T)$ such that for all
$s\in [0,\overline t]$, 
\begin{eqnarray}
  \label{eq:18}
  \frac 3 2 \hat v_n \le \hat \eta(s)\cdot e_n \le  \frac 1 2 \hat v_n,\\
  \label{eq:19}
  \frac {3s} 2 \hat v_n  \le (\hat \xi(s)-\hat x)\cdot e_n \le  \frac s 2 \hat v_n.
\end{eqnarray}
For $t_i\in [0,\overline t]$, we set
\begin{equation}
  \label{eq:20}
  \hat \xi^i (s)= \left\{
    \begin{array}[c]{rcl}
      Q_{t_i,\hat x^i, \hat v^i , \hat \xi(t_i),\hat \eta(t_i)}(s),\quad & \hbox{ if } \quad s \in [0,t_i],\\
      \hat \xi(s),\quad & \hbox{ if } \quad s\in [ t_i,\hat T],
    \end{array}
\right.
\end{equation}
and $\hat \eta^i(s)= \frac {d  \hat \xi^i}   {dt}(s)$ for  $s\in [0,\hat T]$.
Then, we define $\xi^i$ as follows:
\begin{equation}
  \label{eq:21}
   \xi^i (s)= \left\{
    \begin{array}[c]{rcl}
  \Psi_x\left( \hat \xi^ i (s)\right),\quad & \hbox{ if } \quad s \in [0,\hat T],\\
   \xi(s),\quad & \hbox{ if } \quad s\in [\hat T,T],
    \end{array}
\right.
\end{equation}
and $  \eta^i= \frac {d  \xi^i}   {dt}$. Let us first see why  $(\xi^i,\eta^i)\in \Gamma[x_i,v_i]$ for $t_i$ small enough and $i$ large enough.
 A straightforward calculation shows that for $s\le t_i$, 
 \begin{equation}
   \label{eq:22}
   \begin{split}
   \hat\xi^i(s) -\hat x= &(\hat x^i-\hat x) \left(1-3\frac {s^2} {t_i^2} +2 \frac {s^3} {t_i^3}\right)+ s\hat v^i \left(1-2\frac {s} {t_i} + \frac {s^2} {t_i^2}\right)\\ &+ (\hat \xi(t_i)-\hat x)\left(3\frac {s^2} {t_i^2} -2 \frac {s^3} {t_i^3}\right) +s \hat\eta(t_i) \left(-\frac {s} {t_i} + \frac {s^2} {t_i^2}\right).     
   \end{split}
 \end{equation}
Let us focus on $ (\hat \xi^i(s)-\hat x)\cdot e_n= \hat \xi^i(s)\cdot e_n $: from the formula above, we see that 
$ \hat \xi^i(s)\cdot e_n$ is the sum of four terms, the first three of them being nonpositive and the last one nonnegative for all $s\in [0,t_i]$.  Let us consider the sum of the last three terms, namely:
\begin{displaymath}
A(s)=
 s\hat v^i\cdot e_n \left(1-2\frac {s} {t_i} + \frac {s^2} {t_i^2}\right)+
\hat \xi(t_i) \cdot e_n\left(3\frac {s^2} {t_i^2} -2 \frac {s^3} {t_i^3}\right) +s \hat \eta(t_i)\cdot e_n \left(-\frac {s} {t_i} + \frac {s^2} {t_i^2}\right); 
\end{displaymath}
 from (\ref{eq:18}) and (\ref{eq:19}), we see that
\begin{displaymath}
  \begin{split}
      \hat \xi(t_i) \cdot e_n\left(3\frac {s^2} {t_i^2} -2 \frac {s^3} {t_i^3}\right) +s\hat\eta(t_i)\cdot e_n \left(-\frac {s} {t_i} + \frac {s^2} {t_i^2}\right)
 \le & 
    s\hat v \cdot e_n \left(  \frac {1}2 \left(3\frac {s^2} {t_i^2} -2 \frac {s^3} {t_i^3}\right) + 
 \frac {3 }2 \left(-\frac {s} {t_i} + \frac {s^2} {t_i^2}\right)\right)\\ =&  
  s \hat v \cdot e_n \left( -\frac 3 2 \frac {s}{ t_i}   + 3  \frac {s^2} {t_i^2} - \frac {s^3} {t_i^3}  \right)  \end{split}
\end{displaymath}
for all $s\in [0,t_i]$. On the other hand, since $\lim_{i\to \infty} \hat v^i=\hat v$, we see that for $i$ large enough, $ \hat v^i\cdot e_n \le  3  \hat v \cdot e_n /4  $. Hence,
\begin{displaymath}
  \begin{split}
  A(s) 
\le  s  \hat v \cdot e_n \left( \frac 3 4 \left(1- \frac {s}{ t_i}\right)^2  -\frac 3 2 \frac {s}{ t_i}   + 3  \frac {s^2} {t_i^2} - \frac {s^3} {t_i^3}   \right)   
= s  \hat v \cdot e_n \left( \frac 3 4 -3 \frac {s}{ t_i} + \frac{15}{4} \frac {s^2} {t_i^2} - \frac {s^3} {t_i^3}   \right).
  \end{split}
\end{displaymath}
It is easy to check that the function $\theta\mapsto \frac 3 4 -3 \theta + \frac{15}{4} \theta^2  - \theta^3$ is positive for $\theta\in [0,1]$, which implies that $A(s)$ is negative for $s\in [0,t_i]$.
\\
 Hence, for $i$ large enough,  $\hat \xi ^i(s)\cdot e_n \le 0$  for  all $0\le s\le t_i \le \overline t$.\\
On the other hand, since $\lim_{i\to +\infty} \left(|\hat x^i-\hat x|+ |\hat v^i -v|\right)=0$, $\hat \xi$ and $\hat \eta$ are continuous, (\ref{eq:22}) implies that   there exists $I>0$ and $\tilde t \in (0,\overline t\,] $
 such that, if $i\ge I$ and $t_i\in (0, \tilde t)$, then $\hat \xi ^i (s)\in \overline{B_-(0,R_x)}$   for all $s\in [0,t_i]$. Hence, for $i\ge I$ and 
 $t_i\in(0, \tilde t)$, $(\xi ^i,\eta^i)\in \Gamma[x^i, v^i]$.

\medskip

Let us now turn to $\left \|\frac {d\eta^i}{dt}\right\|_{L^2(0,t_i)}$:  straightforward calculus shows that
\begin{displaymath}
 \frac {d\eta^i}{dt}(t)= D\Psi_x(\hat \xi ^i (t)) \frac{d\hat \eta^i}{dt}(t) + \left(D^2\Psi_x(\hat \xi ^i (t))
 \hat \eta^i(t) \right) \hat \eta^i(t).    
\end{displaymath}
This implies that 
\begin{equation}
  \label{eq:23}
  \left\|\frac {d\eta^i}{dt}\right\|^2 _{L^2(0,t_i)} \le   C  \left( \left\|\frac {d\hat\eta^i}{dt}\right\|^2 _{L^2(0,t_i)} + \left\|\hat\eta^i\right\|^4 _{L^4(0,t_i)}  \right  ),
\end{equation}
for a constant $C$ independent of $x\in \partial \Omega$. Hereafter, $C$ may vary from line to line.
\\
First, we focus on $\left\|\frac {d\hat\eta^i}{dt}\right\|^2 _{L^2(0,t_i)}$:
\begin{displaymath}
  \begin{split}
 & \left\|\frac {d\hat \eta^i}{dt}\right\|^2 _{L^2(0,t_i)}
\\= &\int_0^{t_i} \left| 2 \left( 3\frac{ \hat \xi(t_i)-\hat v^it_i -\hat x ^i}{t_i^2} -
\frac {\hat\eta(t_i)-\hat v^i}{t_i} \right)  
+\frac {6s}{t_i} \left( -2\frac{ \hat \xi(t_i)-\hat v^it_i -\hat x ^i}{t_i^2} +
\frac {\hat \eta(t_i)-\hat v^i}{t_i} \right) 
   \right|^2  ds     
\\
\le & 2I_1 +2I_2,
  \end{split}
\end{displaymath}
where 
\begin{displaymath}
  I_1= \int_0^{t_i} \left| 2 \left( 3\frac{ \hat \xi(t_i)-\hat vt_i -\hat x }{t_i^2} -
\frac {\hat\eta(t_i)-\hat v}{t_i} \right)  
+\frac {6s}{t_i} \left( -2\frac{ \hat \xi(t_i)-\hat v t_i -\hat x}{t_i^2} +
\frac {\hat\eta(t_i)-\hat v}{t_i} \right) 
   \right|^2  ds    
\end{displaymath}
and 
\begin{displaymath}
  I_2= \int_0^{t_i} \left| 2 \left( 3\frac{ \hat x-\hat x^i }{t_i^2}+ 2
\frac {\hat v-\hat v^i}{t_i} \right)  
+\frac {6s}{t_i} \left( -2\frac{ \hat x-\hat x^i}{t_i^2} -
\frac {\hat v-\hat v^i}{t_i} \right) 
   \right|^2  ds    .
\end{displaymath}
 Standard arguments yield that
 \begin{displaymath}
   I_1\le C \left \|\frac {d\hat\eta}{dt}\right\|^2_{L^2(0,t_i)},
 \end{displaymath}
for an absolute constant $C>0$. Therefore, given $\epsilon>0$, there exists $\hat t:\;  0< \hat t < \overline t$ such that  $2I_1 <\epsilon/2$ for all $t_i<\hat t$.
\\
On the other hand,
 \begin{displaymath}
   I_2\le C_1   \left( \frac {|\hat x-\hat x^i|^2 }{t_i^3} + \frac {|\hat v-\hat v^i|^2 }{t_i}  \right) 
\le C   \left( \frac {|x- x^i|^2 }{t_i^3} + \frac {| v- v^i|^2 }{t_i}  \right).
 \end{displaymath}
 It is possible to choose the sequence $t_i$ such that 
\begin{itemize}
\item $\lim_{i\to \infty} t_i=0$,
\item $\lim_{i\to \infty} \frac {|x-x^i|^2 }{t_i^3} + \frac {|v-v^i|^2 }{t_i}=0$,
\item $(\xi ^i,\eta^i)\in \Gamma[x^i, v^i]$ for $i$ large enough.
\end{itemize}
Such a choice of $t_i$ yields that $\lim_{i\to \infty} \left\|\frac {d\hat\eta^i}{dt}\right\|_{L^2(0,t_i)}^2=0$.
On the other hand, the choice made on $t_i$ also implies that  $\lim_{i\to \infty} \frac {|x-x^i| }{t_i}=0$, and in turn that $\left\|\hat\eta^i\right\| _{L^\infty(0,t_i)}$ is uniformly bounded with respect to $i$; therefore,  the quantity $ \left\|\hat\eta^i\right\|^4 _{L^4(0,t_i)} $ tends to $0$ as $i$ tends to $\infty$; using (\ref{eq:23}), we have proved that $\lim_{i\to \infty}\left \|\frac {d\eta^i}{dt}\right\|_{L^2(0,t_i)}=0$.
\\
  Therefore, it is possible to choose a sequence $t_i>0$ such that the trajectories $(\xi^i, \eta^i)$ are admissible for $i$ large enough and 
$
\ds \lim _{i\to \infty} \left\|\frac {d\eta^i}{dt} -\frac {d\eta}{dt} \right\| _{L^2(0,T)}=0  
$.
This achieves the proof  in case 2.
\item 
 \begin{enumerate}
 \item $x\in \partial \Omega$, $v\cdot n(x) = 0$ and  $v^i\cdot \nabla d (x ^i)\le 0$ at least for $i$ large enough. We may assume that  $v^i\cdot \nabla d (x ^i)\le 0$ for all $i$. Using the same notation as in case 2,
we see from (\ref{eq:8}) that   $\hat v^i \cdot e_n=  v^i \cdot D\Phi_x^T (x^i) e_n= v^i\cdot \nabla d (x ^i) \le 0$.
\\
 In the present case, the approximate trajectories will have three successive phases, see Remarks~\ref{item:1} and \ref{item:2} for explanations on these different phases.
\\
Given $t_{i,1}\in (0,\hat T)$, we define $(\hat y^i, \hat w^i)$ as follows:
 \begin{equation}
   \label{eq:24}
   \begin{array}[c]{rcl}
   \hat y^i&=& \hat x_n  e_n +    \pi_{e_n^\perp} \left(\hat x^i+ \hat v^i t_{i,1}\right), \\     
   \hat w^i&=& \hat v_n e_n +  \pi_{e_n^\perp} \left( \hat v^i\right)= \pi_{e_n^\perp} \left( \hat v^i\right),     
   \end{array}
 \end{equation}
where $\pi_{e_n^\perp}$ stands for the orthogonal projector on $e_n^\perp=\R^{n-1}\times\{0\}$,
and set 
 \begin{displaymath}
  \hat \xi^i (s)=       Q_{t_{i,1},\hat x^i, \hat v^i , \hat y^i,\hat w^i}(s) \quad \hbox{and }\quad \hat \eta^i(s)= \frac {d  \hat \xi^i}   {ds}(s)\quad \quad \hbox{for }\quad 0\le s\le t_{i,1}.
  \end{displaymath}

\begin{remark}\label{item:1}
In this first phase of the approximate trajectory, i.e. for $s\in [0,t_{i,1}]$, 
 $ \pi_{e_n^\perp }\left(Q_{t_{i,1},\hat x^i, \hat v^i , \hat y^i,\hat w^i}''(s)\right) =0$. The effort only lies in driving the $n^{\rm{th}}$-components of $\hat \xi^i(s)$ and $\hat \eta^i(s)$  so that they match those of $\hat x$ and $\hat v$ at $s=t_{i,1}$. 
  \end{remark}
As above, we first check that for $t_{i,1}$ small enough and $i$ large enough,
$ \hat \xi^i(s)\in\overline{B_-(0,R_x)}$
for all $s\in [0,t_{i,1}]$:
from the definition of $Q_{t_{i,1},\hat x^i, \hat v^i , \hat y^i,\hat w^i}$,
we see that
\begin{equation}
  \label{eq:25}
   \hat \xi^i(s)\cdot e_n= \left(  \hat x^i\cdot e_n \left(1+2 \frac {s} {t_{i,1}} \right)  + s\hat v^i\cdot e_n   \right)
  \left(1-\frac {s} {t_{i,1}}\right)^2
\end{equation}
is nonpositive for $s\in [0,t_{i,1}]$. On the other hand,  we see that 
there exist $I>0$ and $ 0<\tilde t \le \hat T $ such that 
if $i\ge I$ and $0<t_{i,1}< \tilde t$, then for all $s\in [0,t_{i,1} ]$,  $\hat \xi^i(s)\in \overline {B_-(0, R_x)}$.
\\
As in case 2., we need to focus on $\left \|\frac {d\hat \eta^i}{dt}\right\|_{L^2(0,t_{i,1})}$:  straightforward calculus shows that
\begin{displaymath}
  \begin{split}
 \left\|\frac {d\hat \eta^i}{dt}\right\|^2 _{L^2(0,t_{i,1})} 
\le C \left( \frac {  |\hat x^i\cdot e_n|^2 }{t_{i,1}^3} + \frac {|(\hat v-\hat v^i)\cdot e_n|^2 }{t_{i,1}}  \right) 
=  C \left( \frac {  d^2(x^i) }{t_{i,1}^3} + \frac {| v^i\cdot \nabla d (x^i)|^2 }{t_{i,1}}  \right),
  \end{split}
 \end{displaymath}
and we see as above that there exists a sequence $t_{i,1}$ such that 
\begin{itemize}
\item $\lim_{i\to \infty} t_{i,1}=0$,
\item $\lim_{i\to \infty}  \frac {  d^2(x^i) }{t_{i,1}^3} + \frac {| v^i\cdot \nabla d (x^i)|^2 }{t_{i,1}}
=0$,
\item  $  \hat \xi^i(s)\in \overline {B_-(0,R_x)}$ for all $0\le s\le t_{i,1}$. 
\end{itemize}
Taking the derivative of $\hat \xi ^i $ and arguing as in case 2., we also see that $\lim_{i\to \infty} \|\hat \eta^i\|_{L^4(0,t_{i,1})}=0$, because $\lim_{i\to \infty} \frac {d(x^i)}{t_{i,1}}=0$.

\medskip

Next, for $t_{i,1}<t_{i,2}<\hat T$, we set
\begin{equation}
  \label{eq:26}
  \hat \xi^i(s)=\left\{
    \begin{array}[c]{rcl}
       Q_{t_{i,1},\hat x^i, \hat v^i , \hat y^i,\hat w^i}(s) ,\quad & \hbox{ if } &\quad s\le t_{i,1},\\
      Q_{t_{i,2}-t_{i,1}, \hat y^i - \hat x, \hat w^i-\hat v, 0,0 }(s-t_{i,1})+ \hat \xi(s-t_{i,1}), \quad &\hbox{ if }& t_{i,1}\le s\le t_{i,2},
    \end{array}
\right.
\end{equation}
and
\begin{equation}
\label{eq:27}
  \xi^i(s)=\left\{
    \begin{array}[c]{rcl}
\Psi_x\left( \hat \xi^i (s) \right) ,\quad & \hbox{ if } &\quad s\le t_{i,2},\\
\xi(s-t_{i,1}),\quad & \hbox{ if } &\quad  t_{i,2}\le s\le T.
\end{array}\right.
\end{equation}
As above $\hat \eta^i (s) = \frac {d \hat \xi ^i}{dt}(s)$ for $0\le s\le t_{i,2}$ and 
 $\eta^i (s) = \frac {d \xi ^i}{dt}(s)$ for $0\le s\le T$.
\begin{remark}\label{item:2}
 In the second phase of the approximate trajectory, i.e. for $s\in [t_{i,1}, t_{i,2}]$,
the components  of $\hat \xi^i(s)$  and  $\hat \xi(s-t_{i,1})$ parallel to $e_n$  coincide, 
i.e. $ Q_{t_{i,2}-t_{i,1}, \hat y^i - \hat x, \hat w^i-\hat v, 0,0 }(s-t_{i,1})\cdot e_n =0$.
 The effort only consists of driving the 
 projections of $\hat \xi^i(s)$ and $\hat \eta^i(s)$ on $e_n^\perp$ 
 such that they match those of $\hat \xi(s-t_{i,1})$ and $\hat \eta(s-t_{i,1})$ at $s=t_{i,2}$.
 We will see that is not  necessary to have 
$t_{i,2}$ tend to zero, because from the choice of $t_{i,1}$, the distance between $(\hat \xi^i(t_{i,1}),\hat \eta^i(t_{i,1}) )$ and $(\hat x,\hat v)$ tends to $0$ as $i\to +\infty$.  
\end{remark}

It is possible to choose the sequence $t_{i,2}$ bounded from below by a positive constant which depends on $(x,v)$ but not on  $i$  such that $\hat \xi ^ i (s) $ stays in $\overline{B_-(0,R_x)}$ for $s\in  [t_{i,1}, t_{i,2}]$.
Hence, $(\xi^i,\eta^i)\in \Gamma[x^i, v^i]$. 
\\
Moreover, since $t_{i,2}$ is bounded away from $0$ and 
$\lim_{i\to \infty} \left( |\hat y^i -\hat x|  +  | \hat w^i -\hat v| \right)=0$,
it is not difficult to check that 
$\ds
  \lim_{i\to \infty}  \left\|\frac {d\eta^i}{dt} -\frac {d\eta}{dt} \right\| _{L^2(0,T)}=0 
$;
this achieves the proof  in  subcase 3.(a).
 \item $x\in \partial \Omega$, $v\cdot n(x) = 0$,  $v^i\cdot \nabla d(x^i) > 0$ for all  $i$ (or for $i$ large enough), and 
(\ref{eq:12}) holds.
\\
The trajectory $\xi^i$  is constructed as in (\ref{eq:26})-(\ref{eq:27}), but  a further restriction on $t_{i,1}$ is needed in order to guarantee that the trajectory is admissible. Using (\ref{eq:25}), we see that the trajectory is admissible if 
\begin{displaymath}
  \frac{\hat v^i\cdot e_n} {|\hat x^i\cdot e_n |}\le \frac 1{s}+\frac 2 {t_{i,1}},\quad \hbox{ for all } 0\le  s \le t_{i,1}.
\end{displaymath}
This happens if and only if
\begin{displaymath}
  t_{i,1} \le    \frac {3|\hat x^i \cdot e^n |} {\hat v^i\cdot e_n} = \frac {3|d( x^i ) |} {v^i\cdot \nabla d(x_i)}  ,
\end{displaymath}
which should be supplemented with the other two conditions as in  3.(a):
 \begin{eqnarray}
    \label{eq:28}
\lim_{i\to \infty} t_{i,1}=0,
\\ \label{eq:29}
\lim_{i\to \infty} \frac {  |d(x^i)|^2 }{t_{i,1}^3} + \frac {|v^i\cdot \nabla d(x ^i)|^2 }{t_{i,1}}=0.
  \end{eqnarray}
If (\ref{eq:12}) holds, then it is possible to choose such a sequence $t_{i,1}$. Then, as in 3.(a),
 it is possible to choose the sequence $t_{i,2}$ bounded from below by a positive constant independent of $i$ such that $(\xi^i,\eta^i)\in \Gamma[x^i, v^i]$; the last part  of the proof is identical as in subcase 3.(a).
\end{enumerate}
\end{enumerate}
\end{proof}

 \begin{lemma}  \label{sec:clos-graph-prop-7}
   Consider  $(x,v)\in\Xi^{\rm{ad}}$ and a sequence $(x^i, v^i)_{i\in \N}$ such that for all $i\in \N$,   $(x^i, v^i)\in\Xi^{\rm{ad}}$  and $(x^i, v^i)\to (x,v)$ as $i\to \infty$. Suppose  that Assumption \ref{sec:setting-notations-0} and one among the three conditions in Lemma \ref{sec:clos-graph-prop-12}
  are satisfied.
Let  a sequence $(\xi^i, \eta^i)_{i\in \N}$ be such that for all $i\in \N$,    $ (\xi^i, \eta^i)\in \Gamma^{\rm {opt}}[x^i, v^i]$. If $(\xi^i, \eta^i)$ tends to $(\xi,\eta)$ uniformly in $[0,T]$, then  $\eta\in W^{1,2} (0,T;\R^n)$ and 
    $(\xi,\eta)\in \Gamma^{\rm opt} [x,v]$.
 \end{lemma}
 \begin{proof}
   We  need to prove that for any $(\widetilde \xi, \widetilde \eta)\in\Gamma[x,v]$ such that $ \widetilde\eta\in W^{1,2} (0,T;\R^n)$,
     \begin{equation}
     \label{eq:30}
     J( \xi,  \eta,  \eta')\le J(\widetilde \xi, \widetilde \eta, \widetilde \eta ') .
   \end{equation}
   From Lemma \ref{sec:clos-graph-prop-12} applied to $ (\widetilde \xi, \widetilde \eta)$, there exists a sequence $ (\widetilde \xi^i, \widetilde \eta^i)_{i\in \N}$,  with $ (\widetilde \xi^i, \widetilde \eta^i)\in  \Gamma[x^i, v^i]$ such that
   $ (\widetilde \xi^i, \widetilde \eta^i)\to  (\widetilde \xi, \widetilde \eta)$ uniformly on $[0,T]$ as $i\to \infty$, and
   \begin{displaymath}
   \ds \lim_{i\to \infty} \int _0^T \left |\frac {d  \widetilde \eta^i} {dt}(s)\right|^2 ds=
   \int _0^T \left |\frac {d  \widetilde \eta} {dt} (s)\right|^2 ds.     
   \end{displaymath}
   On the other hand, the optimality of $ (\xi^i, \eta^i)$ yields that
   \begin{equation}
     \label{eq:31}
     J\left (\xi^i, \eta^i  ,    \frac {d  \eta^i} {dt}   \right)\leq  J\left (\widetilde \xi^i, \widetilde \eta^i  ,
       \frac {d  \widetilde \eta^i} {dt}   \right).
   \end{equation}
   From the properties of  $ (\widetilde \xi^i, \widetilde \eta^i)$, the right hand side of  \eqref{eq:31} converges to
   $ \ds J\left (\widetilde \xi, \widetilde \eta  ,  \frac {d  \widetilde \eta} {dt}   \right)$.
   The left side  \eqref{eq:31} is thus bounded. Combining this fact with the uniform convergence of
   $(\xi^i, \eta^i)$ to $(\xi,\eta)$  in $[0,T]$, we obtain that the sequence   $ \int _0^T \left |\frac {d   \eta ^i} {dt} (s)\right|^2 ds$ is bounded. This implies that $ \frac {d   \eta ^i} {dt} \rightharpoonup  \frac {d   \eta } {dt}$ in $L^2(0,T;\R^n)$ weakly  and that $\ds \liminf_{i\to \infty} \int _0^T \left |\frac {d   \eta ^i} {dt}(s) \right|^2 ds
   \ge \int _0^T \left |\frac {d   \eta } {dt}(s)\right|^2 ds$. We deduce that
   \begin{displaymath}
      J\left (\xi, \eta  ,    \frac {d  \eta} {dt}   \right)\le \liminf_{i\to \infty}  J\left (\xi^i, \eta^i  ,    \frac {d  \eta^i} {dt}   \right).
   \end{displaymath}
 Combining the information obtained above, we obtain   \eqref{eq:30},     which achieves the proof.
 \end{proof}

  \begin{proof} [Proof of Proposition~\ref{sec:contr-probl-polyg-1}]
   Consider  $(y,w)\in\Theta$ and a sequence $(y^i, w^i)_{i\in \N}$ such that for all $i\in \N$,   $(y^i, w^i)\in\Theta$ and $(y^i, w^i)\to (y,w)$ as $i\to \infty$.
   Consider a sequence     $(\xi^i,\eta^i)_{i\in \N}$ such that for all $i\in \N$, $(\xi^i, \eta^i) \in
 \Gamma^{\rm{opt}}[y^i, w^i]$ and that  $(\xi^i,\eta^i)$ tends to $(\xi,\eta)$
   uniformly. 
   Thanks to the assumption made in the statement of  Proposition~\ref{sec:contr-probl-polyg-1},
   possibly after the extraction of a subsequence, we may suppose that one among the three conditions  
   in Lemma~\ref{sec:clos-graph-prop-12} holds.
   Then the conclusion follows from Lemma  \ref{sec:clos-graph-prop-7}.
 \end{proof}
 \subsection{Bounds related to optimal trajectories}
 \label{sec:bounds1}
 \begin{definition}
   \label{sec:state-constr-optim}
 For a  positive number $C$,  let us set
\begin{eqnarray}
  \label{eq:33}
  K_C&=& \{  (x,v)\in  \Xi: \; |v|\le C\},   
\\
\label{eq:34}
  \Gamma_C &=& \left\{  (\xi,\eta)\in \Gamma: \left|
      \begin{array}[c]{l}
        (\xi(t), \eta(t))\in K_C, \quad \forall t\in [0,T],\\
        \left \| \frac {d\eta}{dt} \right\|_{L^2(0,T;\R^n)}\le C.
      \end{array}
\right.  \right\}.
\end{eqnarray}
 \end{definition}

 \begin{proposition}\label{sec:contr-probl-conv-1}
 Given $r>0$,  let us define 
\begin{equation}
  \label{eq:35}
\Theta_r= \Theta \cap K_r,
\end{equation}
where $K_r$ is defined by (\ref{eq:33}) and $\Theta$ is a closed subset of $\Xi^{\rm ad}$ which satisfies the assumption in Proposition \ref{sec:contr-probl-polyg-1}.
\\
 Under Assumption~\ref{sec:setting-notations-0}, the value function $u$ defined in (\ref{eq:7})
 is continuous  on $ \Theta_r$.\\
There exists a positive number
 $C= C(r, M)  $  such that
if  $(x,v)\in \Theta_r$, then $\Gamma^{\rm opt} [x,v]\subset  \Gamma_C$.
 \end{proposition}
 \begin{remark}
   \label{sec:bounds-relat-optim}
The set of trajectories $\Gamma_C$ is a compact subset of $\Gamma$. In the context of mean field games, see Section \ref{sec:MFG-contr-accel},
the existence of relaxed equilibria will be obtained by applying Kakutani's fixed point theorem to a multivalued map defined on a closed set of probability measures on  $\Gamma_C$.
 \end{remark}

 \begin{proof}
   Take $(x,v)\in \Theta_r$ and a sequence $(x^i,v^i)_{i\in \N}$, $ (x^i,v^i)\in  \Theta_r$ such that 
$\lim_{i\to \infty}   (x^i,v^i)= (x, v)$.
\\
From  Lemma \ref{sec:contr-probl-conv} we know that $u(x,v)$ is finite and from  Assumption \ref{sec:setting-notations-0}, that the infimum in (\ref{eq:7}) is achieved by a trajectory $(\xi,\eta)\in \Gamma^{\rm{opt}}[x,v] $. \\
 Possibly after the extraction of a subsequence, we may assume that $ (x^i,v^i)$ satisfies one among the three points in Lemma~\ref{sec:clos-graph-prop-12}. Then,  there exists a sequence  $(\xi^i,\eta^i)_{i\in \N}$ such that for all $i\in \N$,  $(\xi^i,\eta^i)\in \Gamma [x^i,v^i]$, $\eta^i\in W^{1,2} (0,T;\R^n)$, and  $(\xi^i,\eta^i)$ tends to $(\xi,\eta)$ 
  in $W^{2,2}(0,T;\R^n) \times  W^{1,2}(0,T;\R^n)$, hence uniformly in $[0,T]$. Hence,
  \begin{displaymath}
    \lim_{i\to \infty} J\left(\xi^i,\eta^i , \frac {d\eta^i}{dt}\right)= u(x,v).
  \end{displaymath}
On the other hand,
\begin{displaymath}
J\left(\xi^i,\eta^i , \frac {d\eta^i}{dt}\right) \ge  u(x^i,v^i).  
\end{displaymath}
The latter two observations yield that 
\begin{displaymath}
  \limsup_{i\to \infty} u(x^i,v^i) \le u(x,v).
\end{displaymath}
This proves that $u$ is upper-semi-continuous on $\Theta_r$. From Lemma~\ref{sec:contr-probl-conv}, $u$ is
 continuous on $\Theta_r$. Since $\Theta_r$ is a compact subset of $\R^n\times \R^n$,  $u$ is bounded  on $\Theta_r$.\\
Then from the definition of $J$ and $u$ and the boundedness of $u$ on 
 $\Theta_r$, it is clear that there exists a constant $C=C(r,M)$ such that 
 $\Gamma^{\rm opt}[x,v]\subset \Gamma_C$ for any $(x,v)\in \Theta_r$. 
 \end{proof}

 \section{A mean field game with control on the acceleration  and state constraints}
 \label{sec:MFG-contr-accel}
 \subsection{Setting and notation}
 \label{sec:MFG-contr-accel1}
The bounded domain $\Omega$ of $\R^n$ and the sets $\Xi$ and $\Xi^{\rm ad}$ have been 
introduced in Section~\ref{sec:setting-notation-8}.
 Let $\cP(\Xi)$ be the set of probability measures on $\Xi$.
\\

Let $ C_b^0(\Xi;\R)$ denote the space of bounded and continuous real valued functions defined on  $\Xi$ and 
let $F, G: \cP(\Xi)\to C_b^0(\Xi;\R)$ be  bounded and continuous maps 
(the continuity is  with respect to the narrow convergence  in $ \cP(\Xi)$). 
Let $\ell $ be a real valued, continuous and bounded from below function defined on $\Xi \times [0,T]$. 

Let $F[m]$ and $G[m]$ denote the images by $F$ and $G$  of $m\in \cP(\Xi)$.    
 Set
 \begin{equation}
   \label{eq:37}
M=\max\left( \sup_{(x,v,s)\in \Xi\times[0,T]} \ell_-(x,v,s)+  \sup_{m\in \cP(\Xi)} 
\|F[m]\|_{L^\infty(\Xi)} , \sup_{m\in \cP(\Xi)} \|G[m]\|_{L^\infty(\Xi)}\right) .  
 \end{equation}


Let $\Gamma$ be the set of admissible trajectories given by~(\ref{eq:4}).
It is a metric space with the distance
 $d(  (\xi, \eta), (\tilde\xi, \tilde\eta))= \|\xi-\tilde\xi\|_{  C^1([0,T];\R^n)}$. 
Let $\cP(\Gamma)$ be the set of probability measures on $\Gamma$.\\
For $t\in [0,T]$, the evaluation map $e_t:\Gamma\to \Xi$ is defined by $e_t(\xi, \eta)=  (\xi(t), \eta(t)) $ for all $  (\xi, \eta)  \in \Gamma$.\\
 For any $\mu\in \cP(\Gamma)$, let the Borel probability measure $m^\mu(t)$   on $ \Xi$ be defined by $m^\mu(t)=e_t\sharp \mu$. 
It is possible to prove that if $\mu\in \cP(\Gamma)$, then  $t\mapsto m^\mu(t)$ is continuous from $[0,T]$
 to  $\cP(\Xi)$,
for the narrow convergence  in $ \cP(\Xi)$. Hence, for all $(\xi,\eta)\in \Gamma$,
 $t\mapsto F[m^\mu(t)](\xi(t),\eta(t))$ is continuous and bounded by the constant $M$ in (\ref{eq:37}).
\\
With $\mu\in \cP(\Gamma)$, we associate the cost
\begin{equation}
\label{costMFG}
J^\mu( \xi,\eta  )
=\left (
  \begin{split}
&\int_0^T \left(F[m^\mu(s)]   (\xi(s),\eta(s))+   \ell (\xi(s),\eta(s),s) +\frac 1 2   \left|\frac {d\eta}{dt}(s) \right|^2 \right) ds\\ &+  G[ m^\mu(T)](\xi(T), \eta(T))
  \end{split}\right).
\end{equation}

\begin{remark}\label{sec:setting-notation-7}
It is clear from  (\ref{eq:37}) that given $\mu \in  \cP(\Gamma)$,
 the running cost $  (y,w,s)\mapsto F[m^\mu(s)]   (y,w)+   \ell (y,w,s)$ and the final cost $(y,w)\mapsto G[m^\mu(T)]   (y,w)$ satisfy Assumption \ref{sec:setting-notations-0},
and that the constant arising in (\ref{eq:3})  can be chosen uniformly with respect to $\mu\in   \cP(\Gamma)$.
Hence, for all $\mu \in  \cP(\Gamma)$, Propositions \ref{sec:contr-probl-polyg-1} and \ref{sec:contr-probl-conv-1}
hold for the state constrained control problem related to $J^\mu$ and the constants arising in these propositions can be chosen uniformly with respect to $\mu \in  \cP(\Gamma)$.  
\end{remark}

\begin{assumption}\label{sec:setting-notation-9}
There exists  a positive number $r$ such that  the  
initial distribution of states is a probability measure $m_0$ on $\Xi$ supported in $\Theta_r$, where  $\Theta_r$ is a closed subset of $\Xi^{\rm ad}$ as in (\ref{eq:35}).   
\end{assumption}

 Let $C=C(r,M)$ be the constant appearing in Proposition~\ref{sec:contr-probl-conv-1} (uniform w.r.t. $\mu$), and $\Gamma_C$ be the compact subset of $\Gamma$ defined by (\ref{eq:34}); clearly,  $\Gamma_C$ is a Radon metric space.  From Prokhorov theorem, see \cite[Theorem 5.1.3]{MR2129498}, the set $\cP(\Gamma_C)$ is  compact for the narrow convergence of measures.

Let $\cP_{m_0}(\Gamma)$, (resp. $\cP_{m_0}(\Gamma_C)$) denote the set of probability measures $\mu$ on $\Gamma$
(resp. $\Gamma_C$) such that $e_0\sharp \mu=m_0$.

 Hereafter, we identify $\cP(\Gamma_C)$ with a  subset of $\cP(\Gamma)$ by extending $\mu\in \cP(\Gamma_C)$ by $0$ outside $\Gamma_C$.
Similarly, we may consider  $\cP_{m_0}(\Gamma_C)$ as a subset of $\cP_{m_0}(\Gamma)$.

 Note that for all $\mu\in \cP(\Gamma_C)$ and for all $t\in [0,T]$, $m^\mu(t)$ is supported in $K_C$, where $K_C$ is defined in (\ref{eq:33}).
\begin{remark}
  \label{sec:setting-notation}
Note that $\Gamma_C$ (endowed with the metric of the $C^1\times C^0$-convergence of $(\xi, \eta)$)
 is a Polish space (because it is compact). The multivalued map $\widetilde \Gamma^{\rm{opt}}$ related for instance to $F\equiv 0$ and $G\equiv 0$ maps $\Theta_r$ to non empty and closed subsets of $\Gamma_C$ (the closedness can be checked by usual arguments of the calculus of variations). Since the graph of $\widetilde \Gamma^{\rm{opt}}$ is closed, $\widetilde \Gamma^{\rm{opt}}$ is measurable. Therefore, there exists a measurable selection $j: \Theta_r\to \Gamma_C$ from Kuratowski and Ryll-Nardzewski theorem, \cite{MR188994}. Then $j\sharp m_0$ belongs to  $\cP_{m_0}(\Gamma_C)$. The set $\cP_{m_0}(\Gamma_C)$ is not empty.
\end{remark}

\subsection{Existence of a mean field game  equilibrium}
\label{sec:exist-an-equil}

\begin{lemma}\label{sec:setting-notation-4}
  Let a sequence of probability measures $(\mu_i)_{i\in \N}$,  $\mu_i\in\cP(\Gamma)$,
 be narrowly convergent to $\mu\in \cP(\Gamma)$. For all $t\in [0,T]$, $(m^{\mu_i}(t))_{i\in \N}$ is narrowly convergent to $m^{\mu}(t)$.
\end{lemma}
\begin{proof}
  For all $f\in C^0_b(\Xi;\R)$,
  \begin{displaymath}
    \begin{array}[c]{rcl}
    \ds \int_{\Xi} f(x,v) dm^{\mu_i}(t)(x,v)= \ds \int_{\Gamma} f(\xi(t),\eta(t)) d\mu_i(\xi,\eta)
&\to&\ds  \int_{\Gamma} f(\xi(t),\eta(t)) d\mu(\xi,\eta)\\
&=& \ds\int_{\Xi} f(x,v) dm^{\mu}(t)(x,v).
    \end{array}
  \end{displaymath}
\end{proof}
An easy consequence of Lemma~\ref{sec:setting-notation-4} is that for $C=C(r,M)$ as in Proposition~\ref{sec:contr-probl-conv-1},  $\cP_{m_0}(\Gamma_C)$ is a closed subset of $\cP(\Gamma_C)$, and is therefore compact.

\begin{lemma}
  \label{sec:setting-notation-2}
If $\mu\in \cP(\Gamma_C)$, the map $t\mapsto m^\mu(t)$ is $1/2$-H{\"o}lder continuous from $[0,T]$ to $\cP(K_C)$,  ($K_C$ is defined in (\ref{eq:33}) and $\cP(K_C)$ is endowed with the  Kantorovitch-Rubinstein distance).
\end{lemma}
\begin{proof}
  Let $\phi$ be any Lipschitz function defined on $K_C$  with a Lipschitz constant not larger than $1$. 
  \begin{displaymath}
    \begin{split}
      \int_{K_C} \phi(x,v) \left(dm^\mu(t_2)(x,v)-dm^\mu(t_1)(x,v)\right)&=\int_{\Gamma_C} \left(\phi(e_{t_2}(\xi,\eta))-\phi(e_{t_1}(\xi,\eta))\right) d\mu(\xi, \eta)\\
&=\int_{\Gamma_C} \left(\phi(\xi(t_2),\eta(t_2))-\phi(\xi(t_1),\eta(t_1))\right) d\mu(\xi, \eta)\\
& \le \int_{\Gamma_C} \left(\left|\xi(t_2)  -\xi(t_1)\right|+   \left|\eta(t_2)- \eta(t_1)\right|\right) d\mu(\xi, \eta)\\
& \le C\int_{\Gamma_C} \left(     \left|t_2  -t_1\right|+      \left|t_2- t_1\right|^{\frac 1 2} \right) d\mu(\xi, \eta)\\
& \le \tilde C \left|t_2- t_1\right|^{\frac 1 2},
    \end{split}
  \end{displaymath}
for a constant $\tilde C$ which depends only on $C$ and $T$.
\end{proof}

It is useful to recall the disintegration theorem:
\begin{theorem}
  \label{sec:setting-notation-1}
Let $X$ and $Y$ be Radon metric spaces, $\pi: X\to Y$ be a Borel map,  $\mu$ be a probability measure on $X$. Set $\nu=\pi\sharp\mu$. There exists a $\nu$-almost everywhere uniquely defined Borel measurable family of probability measures $(\mu_y)_{y\in Y}$ on $X$ such that
\begin{equation}
  \label{eq:38}
\mu_y(X\setminus \pi^{-1}(y))=0,\quad \hbox{ for } \nu\hbox{-almost all } y\in Y,
\end{equation}
and for every Borel function $f: X\to [0,+\infty]$,
\begin{equation}
  \label{eq:39}
\int_X f(x) d\mu(x)=\int_Y  \left( \int_{X}f(x) d\mu_y(x) \right) d\nu(y)=\int_Y  \left( \int_{\pi^{-1}(y)}f(x) d\mu_y(x) \right) d\nu(y).
\end{equation}
Recall that $(\mu_y)_{y\in Y}$ is a Borel family  of probability measures if for any Borel subset $B$ of $X$, $Y\ni y\mapsto \mu_y(B)$ is a Borel 
function from $Y$ to $[0,1]$.
\end{theorem}
It is possible to apply Theorem~\ref{sec:setting-notation-1} with $X=\Gamma_C$, $Y=\Theta_r$, $\pi=e_0$ and $\nu=m_0$ (identifying $m_0$ and its restriction to $\Theta_r$): for any $\mu\in \cP_{m_0}(\Gamma_C)$,
there exists  a $m_0$-almost everywhere uniquely defined Borel measurable family of probability measures $(\mu_{(x,v)})_{(x,v)\in \Theta_r}$ on $\Gamma_C$ such that
\begin{equation}
\label{eq:40}
\mu_{(x,v)}( \Gamma_C \setminus e_0^{-1}(x,v))=0,\quad \hbox{ for } m_0\hbox{-almost all } (x,v)\in \Theta_r,
\end{equation}
and for every Borel function $f: \Gamma_C\to [0,+\infty]$,
\begin{equation}
\label{eq:41}
\begin{split}
\int_{\Gamma_C} f(\xi,\eta) d\mu(\xi,\eta)&=\int_{\Theta_r}  \left( \int_{\Gamma_C}f(\xi,\eta) d\mu_{(x,v)}(\xi,\eta) \right) dm_0(x,v)\\ &=
\int_{\Theta_r}  \left( \int_{e_0^{-1}(x,v)}f(\xi,\eta) d\mu_{(x,v)}(\xi,\eta) \right) dm_0(x,v).  
\end{split}
\end{equation}

 For $(x,v)\in \Theta_r$, $m_0$ supported in $\Theta_r$ and $\mu\in \cP_{m_0}(\Gamma_C)$ (where $C=C(r,M)$ is the constant 
appearing in Proposition~\ref{sec:contr-probl-conv-1}), let us set
\begin{equation}
  \label{eq:42}\Gamma^{\mu,{\rm opt}}[x,v]=\left\{ (\xi,\eta)\in \Gamma[x,v]\;:\; J^\mu( \xi,\eta  )=\min_{ ( \widetilde\xi,\widetilde \eta  )\in \Gamma[x,v]}  J^\mu( \widetilde \xi,\widetilde \eta  ) \right\}.
\end{equation}
Standard arguments from the calculus of variations yield that for each $\mu\in \cP_{m_0}(\Gamma_C)$ and $(x,v)\in \Xi^{\rm ad}$, $\Gamma^{\mu,{\rm opt}}[x,v]$ is not empty. Moreover, from Proposition
\ref{sec:contr-probl-conv-1},
 $\Gamma^{\mu,{\rm opt}}[x,v]\subset \Gamma_C$ for all
$(x,v)\in \Theta_r$.
\begin{proposition}
\label{sec:setting-notation-5}
Under the assumptions made  on $\ell$, $F$ and $G$ in paragraph \ref{sec:MFG-contr-accel1},
and Assumption~\ref{sec:setting-notation-9}, let $C=C(r,M)$
be chosen as in Proposition~\ref{sec:contr-probl-conv-1}.
\\
Let a sequence of probability measures $(\mu_i)_{i\in \N}$,  $\mu_i\in\cP_{m_0}(\Gamma_C)$,
 be narrowly convergent to $\mu\in \cP(\Gamma_C)$. Let $(x^i,v^i)_{i\in \N}$, $(x^i, v^i)\in \Theta_r$, 
converge to $(x,v)$.  Consider a sequence     $(\xi^i,\eta^i)_{i\in \N}$ such that for all $i\in \N$, $(\xi^i, \eta^i) \in \Gamma^{\mu_ i,\rm{opt}}[x^i, v^i]$. If $(\xi^i,\eta^i)_{i\in \N}$ tends to $(\xi,\eta)$
uniformly,  then $(\xi,\eta)\in \Gamma^{\mu,\rm{opt}}[x, v]$. In other words, the multivalued map $(x,v,\mu)\mapsto \Gamma^{\mu,\rm{opt}}[x, v]$ has closed graph.
\end{proposition}
\begin{proof}
First, from Lemma~\ref{sec:setting-notation-4}, $\mu\in \cP_{m_0}(\Gamma_C)$ and for all $t\in [0,T]$, $(m^{\mu_i}(t))_{i\in \N}$ is narrowly convergent to $m^{\mu}(t)$. From the continuity assumptions made on $F$ and $G$ and the dominated convergence theorem, we deduce that 
\begin{displaymath}
  \begin{split}
    \int_0^T F[m^{\mu_i}(t)](\xi^i(t),\eta^i(t) ) dt\quad &\to \quad  \int_0^T F[m^{\mu}(t)](\xi(t),\eta(t) ) dt,\\
G[m^{\mu_i}(T)](\xi^i(T),\eta^i(T) ) &\to \quad   G[m^{\mu}(T)](\xi(T),\eta(T) ) .
  \end{split}
\end{displaymath}
The last part  of the proof is completely similar to the proof of Proposition \ref{sec:contr-probl-polyg-1}.
It makes use of Assumption~\ref{sec:setting-notations-0} and Lemma \ref{sec:exist-an-equil}.
\end{proof}
\begin{definition}
  \label{sec:setting-notation-3}
The probability measure $\mu \in \cP_{m_0}(\Gamma)$ is a constrained mean field game equilibrium associated with the initial distribution $m_0$ if 
\begin{equation}\label{eq:43}
  {\rm{supp}}(\mu)\subset \mathop \bigcup_{
(x,v)\in   {\rm{supp}}(m_0)      }\Gamma^{\mu,{\rm opt}}[x,v].
\end{equation}
\end{definition}

\begin{theorem}
  \label{sec:setting-notation-6}
Under the assumptions made  on $F$ and $G$ at the beginning of paragraph~\ref{sec:MFG-contr-accel1}
and Assumption~\ref{sec:setting-notation-9}, let $C=C(r,M)$
be chosen as in Proposition~\ref{sec:contr-probl-conv-1}.
There exists a constrained mean field game equilibrium $\mu\in \cP_{m_0}(\Gamma_C)$, see Definition~\ref{sec:setting-notation-3}. Moreover, $t\mapsto e_t\sharp \mu \in C^{1/2} ([0,T];\cP(K_C))$,  ($K_C$ is defined in (\ref{eq:33}) and $\cP(K_C)$ is endowed with the  Kantorovitch-Rubinstein distance).
\end{theorem}

\begin{proof}
The proof follows that of Cannarsa and Capuani in \cite{MR3888967}.
Define the multivalued map $E$ from $\cP_{m_0}(\Gamma_C)$ to  $\cP_{m_0}(\Gamma_C)$ as follows: for any $\mu\in \cP_{m_0}(\Gamma_C)$,
\begin{equation}
  \label{eq:44}
E(\mu)=\left\{ \hat \mu \in \cP_{m_0}(\Gamma_C)\;:\; {\rm{supp}}(\hat \mu_{(x,v)})\subset\Gamma^{\mu, {\rm{opt}}}[x,v] \hbox{ for } m_0\hbox{-almost all }(x,v)\in \Xi \right\},
\end{equation}
where $(\hat \mu_{(x,v)})_{(x,v)\in  \Xi}$ is the $m_0$-almost everywhere uniquely defined Borel measurable family of probability measures which disintegrates $\hat \mu$, see the lines after Theorem~\ref{sec:setting-notation-1}. 
\\
Then the measure $\mu\in \cP_{m_0}(\Gamma_C)$ is a constrained mean field game equilibrium if and only if  $\mu\in E(\mu)$. This leads us to  apply Kakutani fixed point  theorem to the multivalued map $E$, see \cite{MR2378491,MR46638}. Several steps are needed in order to check that the assumptions of Kakutani theorem are satisfied. First of all, we recall that $\cP_{m_0}(\Gamma_C)$ is compact.
\begin{description}
\item[Step 1:] For any $\mu\in \cP_{m_0}(\Gamma_C)$, $E(\mu)$ is a non empty convex set.\\
First,  we have already seen   that $\Gamma^{\mu, {\rm{opt}}}[x,v]\not= \emptyset$ and that the map $(x,v)\mapsto \Gamma^{\mu,\rm{opt}}[x, v]$ has closed graph.
Therefore, from \cite{MR1048347}, $(x,v)\mapsto \Gamma^{\mu,\rm{opt}}[x, v]$ has a Borel measurable selection $(x,v)\mapsto (\xi^\mu_{(x,v)}, \eta^\mu_{(x,v)})$. The measure $\hat \mu$ defined by 
\begin{displaymath}
  \hat \mu (B)= \int_{\Theta_r} \delta_{ (\xi^\mu_{(x,v)}, \eta^\mu_{(x,v)})}(B)dm_0(x,v), \quad \hbox{for all Borel subset $B$ of $\Gamma_C$,}
\end{displaymath}
 belongs to $E(\mu)$; indeed, the total mass of $\hat \mu$  is one because $m_0$ is supported in $\Theta_r$ and $C=C(r,M)$
as in Proposition~\ref{sec:contr-probl-conv-1} so $E(\mu)$ is non empty.\\
Second, take $\mu^1$, $\mu^2$ in $E(\mu)$ and $\lambda\in [0,1]$. We wish to prove that $\lambda \mu_1+(1-\lambda)\mu_2\in E(\mu)$. It is clear that $\lambda  \mu^1 +(1-\lambda)\mu^2$ belongs to $\cP_{m_0}(\Gamma_C)$. On the other hand, since $\mu^1$ belongs to $E(\mu)$, there exist a $m_0$-almost everywhere uniquely defined Borel measurable family
$ (\hat \mu^1_{(x,v)})_{(x,v)\in  \Theta_r}$ of probability measures which disintegrates $\mu^1$ and a subset $A^1$ of $\Theta_r$ such that $m_0(A^1)=0$ and $\hbox{supp}(\mu^1_{(x,v)})\subset  \Gamma^{\mu,\rm{opt}}[x, v]$ for all $(x,v)\in \Theta_r\setminus A^1$.
Similarly, $\mu^2$ can be disintegrated into a $m_0$-almost everywhere uniquely defined Borel measurable family
$ (\hat \mu^2_{(x,v)})_{(x,v)\in  \Theta_r}$ of probability measures, and there exists a subset $A^2$ of $\Theta_r$ such that $m_0(A^2)=0$ and $\hbox{supp}(\mu^2_{(x,v)})\subset  \Gamma^{\mu,\rm{opt}}[x, v]$ for all $(x,v)\in \Theta_r\setminus A^2$. Therefore,
$\lambda  \mu^1 +(1-\lambda)\mu^2$ can be disintegrated as follows: for each Borel function $f$ defined on $\Gamma_C$,
\begin{displaymath}
  \begin{split}
&\int_{\Gamma_C} f(\xi,\eta) d\left(\lambda\mu^1+(1-\lambda\right)\mu^2)(\xi,\eta)]
\\
=
&
\int_{\Theta_r}  \left( \int_{\Gamma_C}f(\xi,\eta) d\left( \lambda \mu^1_{(x,v)}+(1-\lambda)  \mu^2_{(x,v)}\right) (\xi,\eta) \right) dm_0(x,v),\\
&\hbox{supp}\left(\lambda \mu^1_{(x,v)}+(1-\lambda)  \mu^2_{(x,v)}\right)\subset  \Gamma^{\mu,\rm{opt}}[x, v],\quad
 \quad \forall (x,v)\in \Theta_r\setminus( A^1\cup A^2 ),
\end{split}
\end{displaymath}
and $m_0(A^1\cup A^2)=0$. Hence, $\lambda  \mu^1 +(1-\lambda)\mu^2\in E(\mu)$, so $E(\mu)$ is convex.
\item[Step 2:] The multivalued map $E$ has closed graph.\\
Consider a sequence $(\mu^i)_{i\in \N}$, $\mu^i\in \cP_{m_0}(\Gamma_C)$ narrowly convergent to  $\mu\in \cP_{m_0}(\Gamma_C)$. 
Let a sequence  $(\hat\mu^i)_{i\in \N}$, $\hat\mu^i\in E(\mu^i)$ be narrowly convergent to $\hat\mu\in \cP_{m_0}(\Gamma_C)$.
 We claim that  $\hat\mu\in E(\mu)$.\\
First, there  exists  a $m_0$-almost everywhere uniquely defined Borel measurable family of probability measures $(\hat \mu_{(x,v)})_{(x,v)}$ on $\Gamma_C$ such that (\ref{eq:40}) and (\ref{eq:41}) hold for $\hat \mu$ and $\hat \mu_{(x,v)}$.
In particular, there exists a subset $A$ of $\Theta_r$ with $m_0(A)=0$ such that for $(x,v)\in \Theta_r\setminus A$, 
$\hat\mu_{(x,v)}( \Gamma_C \setminus e_0^{-1}(x,v))=0$.
\\
Take $(x,v)\in \Theta_r\setminus A$ and $(\hat \xi,\hat \eta)\in {\rm {supp}}(\hat\mu_{(x,v)})$.
\\
 The Kuratowski convergence theorem applied to $(\hat\mu^i)_i,\, \hat \mu$, see \cite{MR0233396}, implies that there exists a sequence $(\hat \xi^i,\hat \eta^i)_{i\in \N}$, $(\hat \xi^i,\hat \eta^i)\in {\rm{supp}} (\hat\mu^i)$, which converges to $(\hat \xi,\hat \eta)$ uniformly in $[0,T]$. Set $(x^i,v^i)=(\hat \xi^i(0),\hat \eta^i(0))\in \Theta_r$. Since $\hat \mu^i\in E(\mu^i)$, there holds that $ (\hat \xi^i,\hat \eta^i)\in \Gamma^{\mu^i, {\rm{opt}}}[x^i,v^i]$. From Proposition~\ref{sec:setting-notation-5}, we see that $ (\hat \xi,\hat \eta)\in \Gamma^{\mu, {\rm{opt}}}[x,v]$. Since $(x,v)$ is any point in $ \Theta_r\setminus A$, this implies that $\hat \mu \in E(\mu)$.
\end{description}
All the assumptions of Kakutani theorem are satisfied: hence, there exists $\mu\in \cP_{m_0}(\Gamma_C)$ such that $\mu\in E(\mu)$. This achieves the proof.
\end{proof}

\begin{definition}\label{mild}
A pair $(u,m)$, where $u$ is a measurable function defined on $ \Xi \times [0,T]$ and $m \in C^0([0,T]; \cP(\Xi))$, is called {\it {a mild solution of the mean field game}}, if there exists a constrained mean fied game equilibrium $\mu$ for $m_0$ (see Definition~\ref{sec:setting-notation-3}) such that:
\begin{itemize}
\item[i)]
$m(t)=e_t\sharp \mu$;
\item[ii)] $\forall (x,v)\in \Xi^{\rm ad}$, $u( x,v,t)$ is given by
  \begin{displaymath}
u(x,v,t)= \inf_{(\xi, \eta, \alpha) \in \Gamma[x,v,t]}
\left(
    \begin{split}
     & \int_t^T
\left(F[m(s)] (\xi(s),\eta(s) )+ \ell(\xi(s),\eta(s), s)
  +
\frac 1 2 |\alpha(s)|^2  \right) ds\\ & +  G[m(T)](\xi(T), \eta(T))    
    \end{split}\right).
  \end{displaymath}
where $\Gamma[x,v,t]$ is the set of admissible trajectories starting from $(x,v)$ at $s=t$.
\end{itemize}
\end{definition}
A corollary of Theorem \ref{sec:setting-notation-6} is:
\begin{corollary}\label{sec:exist-an-equil-1}
Under the  assumptions of  Theorem \ref{sec:setting-notation-6},
there exists a mild solution $(u,m)$. Moreover, $m\in C^{\frac 1 2}([0,T];\cP(K_C))$.
\end{corollary}

\begin{remark}\label{sec:existence-mean-field}
Under classical monotonicity assumptions for $F$ and $G$, see e.g. \cite{MR3888967}, 
the mild solution is unique.
\end{remark}

 \subsection{Non quadratic running costs}
 \label{sec:optim-contr-accel2_3}
It is possible to generalize the results of Sections~\ref{sec:state-constr-optim-1} and~\ref{sec:MFG-contr-accel}  to costs of the form
 \begin{equation}
\label{cost3}
J(\xi,\eta, \alpha)
=\int_0^T \left(\ell(\xi(s),\eta(s) ,s) +\frac 1 p   |\alpha|^p(s) \right) ds+g(\xi(T), \eta(T)),
\end{equation}
where  $1< p$, for dynamics given by (\ref{eq:1}) and staying in $\Xi$.
\\
For brevity, we  restrict ourselves to the closed graph result, whose proof is completely similar to that of 
Proposition~\ref{sec:contr-probl-polyg-1}. The generalization of Theorem \ref{sec:setting-notation-6} is then possible.  


\begin{proposition}\label{sec:clos-graph-prop-1_weaker_assumption3}
  Consider a closed subset $\Theta$ of $\Xi^{\rm ad}$.
Assume that  all sequence  $(x^i, v^i)_{i\in \N}$  such that for all $i\in \N$, $ (x^i, v^i)\in \Theta$ and $\lim_{i\to +\infty} (x^i, v^i)=(x, v)\in \Theta$, the following holds:
if  $x\in  \partial \Omega$,  then 
\begin{equation}
  \label{eq:45}
((v^i \cdot \nabla d (x^i) )_+)^{2p-1} = o\left( \left |d(x^i)\right|^{p-1} \right) ;
\end{equation}
then  the graph of the multivalued map $ \Gamma^{\rm opt}: \; \Theta \rightarrow \Gamma$, $(x,v) \mapsto \Gamma^{\rm opt}[x,v] $
is closed in the sense given in Proposition \ref{sec:contr-probl-polyg-1}.
\end{proposition}

\section{One dimensional problems: more accurate results}
\label{sec:more-accur-results}
In dimension one and for a running cost quadratic in $\alpha$, it is possible to obtain more accurate results under a slightly stronger assumption
on the running cost, namely that it does not favor the trajectories which exit the domain. 
 In particular, the closed graph property can be proved to hold on the whole set 
$\Xi^{\rm {ad}}$, and concerning mean field games, no assumptions are needed on the support of $m_0$ by contrast with Theorem \ref{sec:setting-notation-6}.

\subsection{ Optimal control problem in an interval: a closed graph property}
\label{sec:optim-contr-probl}
 In this paragraph, we set $\Omega=(-1,0)$ and $\Xi=[-1,0]\times \R$.
The optimal control problem consists of minimizing $J(\xi, \eta, \eta')$ given by (\ref{eq:2}) on the dynamics given by (\ref{eq:1}) and staying in $\Xi$.\\
The definition of $\Xi^{\rm{ad}}$ is then modified as follows:
  \begin{equation}
    \label{eq:49}
\Xi^{\rm{ad}}=\Xi \setminus    \Bigl(\{0\}\times(0,+\infty)  \cup  \{-1\}\times(-\infty,0)  \Bigr) .
\end{equation}

We make the following assumptions:
 \begin{assumption}
   \label{sec:optim-contr-probl-2}
   The running cost $\ell: \Xi\times [0,T]\to \R$ is a continuous function, bounded from below.
   The terminal cost $g: \Xi\to \R$ is also assumed  continuous and bounded from below.
   Set $M= \|g_-\|_{L^\infty (\Xi) }  +   \|\ell_-\|_{L^\infty( \Xi\times [0,T]) } $.
 \end{assumption}
\begin{assumption}
  \label{sec:setting-notations-1bis}
For all $t\in [0,T]$ and $v>0$, 
\begin{equation}
  \label{eq:50}
\ell( 0,v,t)\ge \ell( 0,0,t),\quad\quad \hbox{and}\quad \quad \ell( -1,-v,t)\ge \ell( -1,0,t).
\end{equation}
\end{assumption}

An interpretation of Assumption~\ref{sec:setting-notations-1bis} is that the running cost $\ell $ 
penalizes (or at least does not favor) the trajectories that exit $\Xi^{\rm{ad}}$.  In that respect, Assumption~\ref{sec:setting-notations-1bis} is rather natural.

\medskip

For $(x,v)\in \Xi$, let $\Gamma$, $\Gamma[x,v]$ and $\Gamma^{\rm opt} [x,v]$ be defined as follows:
\begin{eqnarray*}
\label{eq:51}
\Gamma=\left\{
  \begin{array}[c]{ll}
(\xi, \eta)\in C^1([0,T];\R) \times AC([0,T];\R) \;:\; &\left|
\begin{array}[c]{ll}
  \xi'(s)=\eta(s), \; &\forall s\in [0,T] ,
  \\   (\xi(s),\eta(s))\in \Xi,  &\forall s\in [0,T] 
\end{array}\right.  \end{array}
\right\},
\\
\Gamma[x,v]=\{ (\xi, \eta) \in \Gamma:\,\xi(0)=x,\, \eta(0)=v\},
\\
\Gamma^{\rm{opt}}[x,v]= {\rm{argmin}}_{(\xi, \eta)\in \Gamma[x,v]  } J(\xi, \eta,\eta') .
\end{eqnarray*}

\begin{theorem}\label{sec:optim-contr-probl-3}
  Under Assumptions \ref{sec:optim-contr-probl-2} and   \ref{sec:setting-notations-1bis}, the graph of the multivalued map $ \Gamma^{\rm opt}: \; \Xi^{\rm ad}\rightarrow \Gamma$,  $(x,v) \mapsto \Gamma^{\rm opt}[x,v]     $,
is closed, in the following sense:
 consider  a sequence $(x^i, v^i)_{i\in \N}$,  $(x^i, v^i)\in\Xi^{\rm{ad}}$,
 such that $\ds \lim_{i\to \infty}(x^i,v^i) =(x,v)\in \Xi^{\rm{ad}}$. Consider a sequence     $(\xi^i,\eta^i)_{i\in \N}$
 such that for all $i\in \N$, $(\xi^i, \eta^i) \in \Gamma^{\rm{opt}}[x^i, v^i]$.
\\
If $(\xi^i,\eta^i)$ tends to $(\xi,\eta)$
uniformly,  then $(\xi,\eta)\in \Gamma^{\rm{opt}}[x, v]$.
\end{theorem}
\begin{remark}
  \label{sec:optim-contr-probl-1}
Note that, by contrast with Proposition \ref{sec:contr-probl-polyg-1}, Theorem \ref{sec:optim-contr-probl-3} holds
for $\Gamma^{{\rm opt}}$ and not only its  restriction to a subset $\Theta$ of $\Xi^{{\rm ad}}$ satisfying suitable conditions.
Hence,  Theorem \ref{sec:optim-contr-probl-3} is more accurate. On the other hand, it requires an additional assumption, namely Assumption \ref{sec:setting-notations-1bis}.\\
Note also that the result stated in Theorem \ref{sec:optim-contr-probl-3}, namely the closed graph property of the multivalued map  $\Gamma^{\rm{opt}}$,
is obtained despite the fact that the value function of the optimal control problem is not continuous and not locally bounded  on $\Xi^{{\rm ad}}$.  This may seem surprising at first glance. Besides, the fact that the value function is singular at some points of  $\Xi^{{\rm ad}}$ will be an important difficulty in the proofs.
\end{remark}
The proof of  Theorem~\ref{sec:optim-contr-probl-3}
relies on several lemmas.
\begin{lemma}
\label{sec:optim-contr-probl-4}
Consider  $(x,v)\in\Xi^{\rm{ad}}$, $(\xi,\eta)\in \Gamma[x,v]$ such that $\eta\in W^{1,2} (0,T;\R)$ and a sequence $(x^i, v^i)_{i\in \N}$ such that for all $i\in \N$,   $(x^i, v^i)\in\Xi^{\rm{ad}}$ 
and $(x^i, v^i)\to (x,v)$ as $i\to \infty$.
\\
If  one among the following assumptions is satisfied,
\begin{enumerate}
\item $x\in \Omega $
\item $x=0$, $v\le 0$ and for all integer $i$,   $v^i\le 0$
\item $(x,v)=(0,0)$,  $v^i>0$ for all integer $i$ and $\ds \lim_{i\to 0} \frac { (v^i)^3}{|x^i|}=0$
\item $x=-1$, $v\ge 0$ and for all integer $i$,   $v^i\ge 0$
\item $(x,v)=(-1,0)$,  $v^i<0$ for all integer $i$ and $\ds \lim_{i\to 0} \frac { |v^i|^3}{|x^i+1|}=0$
\end{enumerate}
then 
there exists a sequence  $(\xi^i,\eta^i)_{i\in \N}$ such that for all $i\in \N$,  $(\xi^i,\eta^i)\in \Gamma [x^i, v^i]$, $\eta^i\in W^{1,2} (0,T;\R)$, and  $(\xi^i,\eta^i)$ tends to $(\xi,\eta)$ 
in $W^{2,2}(0,T;\R) \times  W^{1,2}(0,T;\R)$, hence uniformly in $[0,T]$.
\end{lemma}
\begin{proof}
 Lemma \ref{sec:optim-contr-probl-4} is the  counterpart of Lemma \ref{sec:clos-graph-prop-12}
. The proof is quite similar, so we skip it for brevity.
\end{proof}

 \begin{corollary}\label{sec:clos-graph-prop-1} 
   Consider  $(x,v)\in\Xi^{\rm{ad}}$ and a sequence $(x^i, v^i)_{i\in \N}$ such that for all $i\in \N$,   $(x^i, v^i)\in\Xi^{\rm{ad}}$  and $(x^i, v^i)\to (x,v)$ as $i\to \infty$. Suppose  that Assumption
\ref{sec:optim-contr-probl-2} and one among the five conditions  in Lemma 
 \ref{sec:optim-contr-probl-4} are satisfied.
Let  a sequence $(\xi^i, \eta^i)_{i\in \N}$ be such that for all $i\in \N$,    $ (\xi^i, \eta^i)\in \Gamma^{\rm {opt}}[x^i, v^i]$. If $(\xi^i, \eta^i)$ tends to $(\xi,\eta)$ uniformly in $[0,T]$, then  $\eta\in W^{1,2} (0,T;\R)$ and 
    $(\xi,\eta)\in \Gamma^{\rm opt} [x,v]$.
  \end{corollary}
  \begin{proof}
   Corollary \ref{sec:clos-graph-prop-1} is the counterpart of Lemma~\ref{sec:clos-graph-prop-7}. The proof is identical.
  \end{proof}

 Consider  $(x,v)\in\Xi^{\rm{ad}}$ and a sequence $(x^i, v^i)_{i\in \N}$ such that for all $i\in \N$,   $(x^i, v^i)\in\Xi^{\rm{ad}}$  and $(x^i, v^i)\to (x,v)$ as $i\to \infty$. 
 Because it is always possible to extract subsequences, we can say  that the only cases that have not yet been addressed in  Lemma \ref{sec:optim-contr-probl-4} are the following:
 \begin{equation}
   \label{eq:52}
  \left\{
    \begin{array}[c]{l}
     (x,v)=(0,0),\quad   v^i>0,  \\ \hbox{ and there exists a constant $C>0$ s. t. for all
      $i\in \N$, }
  \frac { (v^i)^3}{|x^i|}\ge C,
    \end{array}
\right.
\end{equation}
or
 \begin{equation}
   \label{eq:5019bis}
  \left\{
    \begin{array}[c]{l}
     (x,v)=(-1,0),\quad   v^i<0,  \\ \hbox{ and there exists a constant $C>0$ s. t. for all
      $i\in \N$, }
  \frac { |v^i|^3}{|x^i+1|}\ge C. 
    \end{array}
\right.
\end{equation}
Since the two cases are symmetrical, we may concentrate on (\ref{eq:52}).
\\
 It is clear that \eqref{eq:52} implies that $ |x^i|/ v^i\to 0$  as $i\to +\infty$,
 because $v^i \to 0$.  In the case when   \eqref{eq:52} is satisfied, we need two  technical lemmas which provide a lower bound for the cost $\int_0^T \left|\frac {d\eta^i}{dt} (s) \right| ^2 ds$ of the admissible trajectories starting at $(x^i, v^i)$:

\begin{lemma} \label{sec:clos-graph-prop-4}
  Consider $(x, v)\in \Xi^{\rm ad}$ such that  $x<0$, $v> 0$, $ 3 {|x|}/ v< T$,   and
  $\theta\in \left( 0 ,T\right)$. 
   Given  a real number $w\in \left[0,  {|x|} / T\right]$, set
   \begin{equation}
   \label{eq:53}
     K_{\theta,w}=\left \{ \eta\in W^{1,2}(0,\theta;\R): \left|
       \begin{array}[c]{l}
         \eta(0)=v,\quad \eta(\theta)= w,\\
         \eta(s)\ge    w,\;   \forall s\in [0,\theta],\\
       \ds   x+\int_0^\theta \eta(s) ds \le 0
       \end{array}\right. \right\}.
 \end{equation}
 The quantity
   \begin{equation}
     \label{eq:54}
     I(\theta,w)= \inf_{\eta\in K_{\theta,w} }  \frac 1 2 \int_0^\theta  \left|\frac {d\eta}{dt} (s) \right| ^2 ds 
   \end{equation}
is achieved by a function $\eta=\eta_{\theta, w}$ and is given by
   \begin{equation}\label{eq:55}
     I(\theta,w)= \left\{
       \begin{array}[c]{rl}
         \ds    \frac 1 2 \frac {(w-v)^2}\theta ,  \quad \quad\quad &\hbox{if}\quad  \theta\in \left
                                                                      [0,  \frac {2|x|}{v+w}\right],\\
         \ds 6 \frac {x^2} {\theta^3} + 6 \frac {x(v+w)}{\theta^2} +2 \frac { v^2+vw+w^2}\theta
          \quad \quad\quad &\hbox{if}\quad  \theta\in \left[ \frac {2|x|}{v+w},  \frac {3|x|}{v+2w}  \right],\\
         \ds \frac 2 9 \, \frac { ( v-w)^3} {|x|-w\theta} , \quad\quad \quad &\hbox{if}\quad  \theta \in  \left[  \frac {3|x|}{v+2w}  ,T\right).
  \end{array}\right.
   \end{equation}
 \end{lemma}

 \begin{remark} \label{sec:clos-graph-prop-6}
The partition of the interval $[0,T]$  in (\ref{eq:55}) is justified by the assumptions of Lemma~\ref{sec:clos-graph-prop-4}.
Indeed
\begin{itemize}
\item  $  {3|x|}/{v}< T$ and $w\ge 0$ imply that  $  {3|x|}/{(v+2w)}< T$
\item $    {2|x|}/{(v+w)}<  {3|x|}/{(v+2w)}$ because $0\le w\le {|x|} /T < v/ 3$.
\end{itemize}
Note also that  if $|x|/v\to 0$, then  $   {3|x|}/{(v+2w)} \sim  {3 |x|}/ v \ll T$.
  \end{remark}
  \begin{proof}
   Problem  \eqref{eq:54} is the minimization of a strictly convex and continuous functional under linear and continuous constraints, and the set $K_{\theta,w}$ is non empty, as we shall see below, convex and closed. Hence there exists a unique minimizer, named  $\eta$ again. The Euler-Lagrange necessary conditions read as follows: there exists a  real number $\mu\ge 0$ such that $\eta $ is a weak solution of the linear complementarity problem (variational inequality)
   \begin{equation}
     \label{eq:56}
   \left\{  \begin{array}[c]{rcll}
       -\eta''&\ge& -\mu, &\hbox{ in } (0, \theta) ,      \\
       \eta &\ge&  w ,  &\hbox{ in }  (0, \theta),  \\  
       (-\eta''+ \mu)( \eta- w) &=&0   &\hbox{ in } (0, \theta) ,      \\ 
\ds         x+\int_0^\theta \eta(s) ds &\le& 0, \\\mu&\ge& 0,\\
       \ds     \mu\left( x+\int_0^\theta \eta(s) ds\right) &=&0,\\
       \eta(0)&=&v,\\
       \eta(\theta)&=&w.
     \end{array}     \right.
 \end{equation}
 The solution of   \eqref{eq:56} can be written explicitly. Skipping the details, it has the following  form:
 \begin{enumerate}
 \item If $\theta\ge  {3|x|}/{(v+2w)}$, then 
   \begin{equation}
     \label{eq:57}
     \left\{ \begin{array}[c]{rcll}
               \eta(t)&=& \ds v - \mu\tau t +\frac \mu 2  t ^2,   \quad & 0\le t \le \tau,\\
               \eta(t)&=&  \ds  w,  \quad & \tau < t \le \theta, 
     \end{array}\right.
 \end{equation}
 with
 \begin{equation}
  \label{eq:58}  \tau= -3 \frac { x+w\theta}{v-w} \quad \hbox{and }\quad  \mu=   \frac { 2(v-w)^3} {9 (x+w\theta)^2}.
 \end{equation}
 Note that $ -3 \frac { x+w\theta}{v-w}\le \theta$ because   $\theta\ge \frac {3|x|}{v+2w}$. Note also that  $\ds x+\int_0^\theta \eta(s) ds=0$. 
We see that $ \ds  I(\theta,w)=  \frac {\mu^2} 2 \int_0 ^ \tau  (-\tau + t)^2  dt=  \frac {\mu^2 \tau^3} 6  =\frac 2 9 \frac {(v-w)^3} {|x|-w\theta}$; we have obtained the third line in (\ref{eq:55}).
  \item     If $ {2|x|}/{(v+w)}\le \theta \le   {3|x|}/{(v+2w)}$, then  for all $t\in [0,\theta]$,
     \begin{equation}
     \label{eq:5024ter}
     \eta(t)= \ds v + k t  +\frac \mu 2  t ^2 ,
   \end{equation}
   with
    \begin{equation}
      \label{eq:5025ter}
      k= -\frac{6x+(4v+2w) \theta}{\theta^2},\quad \hbox{ and }\quad \mu =6 \frac {2x+(v+w)\theta} {\theta^3}. 
    \end{equation}
    Note that $ \ds x+\int_0^\theta \eta(s) ds=0$. Easy algebra leads to   $I(\theta,w)= 6 \frac {x^2} {\theta^3} + 6 \frac {x(v+w)}{\theta^2} +2 \frac { v^2+vw+w^2}\theta$; we have obtained the second line in (\ref{eq:55}).
\item If $\theta \le  {2|x|}/{(v+w)}$, then for all $t\in [0,\theta]$,
   \begin{equation}
     \label{eq:5024bis}
     \eta(t)= \ds v -  (v-w) \frac t \theta .
   \end{equation}
     Then, $      I(\theta,w)=\frac 1 2 \frac {(w-v)^2}\theta$; we have obtained the first line in (\ref{eq:55}).
       Note that if  $\theta < \frac {2|x|}{v+w}$, then $ \ds x+\int_0^\theta \eta(s) ds<0$. 
 \end{enumerate}
\end{proof}
\begin{lemma} \label{sec:clos-graph-prop-5}
  Consider a sequence  $(x^i, v^i)_{i\in \N}$  such that  $x^i<0$, $v^i> 0$ for all  $i\in \N$, and  $v^i\to 0$,  $ |x^i|/ v^i\to 0$  as $i\to +\infty$. Call $I^i(\theta,w)$ the quantity given by \eqref{eq:54} for $v=v^i$, $x=x^i$ and $w\in \left[0,  {x^i} /T\right]$. Then
  \begin{equation}
    \label{eq:59}
    \inf   \left \{I^i(\theta,w),\; \theta \in  \left(0 ,T\right)\right\}
  = \frac 2 9 \frac {(v^i)^3} {|x^i|} +o(1),
  \end{equation} where $o(1)$ is a quantity that tends to $0$ as $i$ tends to infinity (which is in fact of the order of $(v^i)^2$ or smaller). 
\end{lemma}
\begin{proof}
  Recall that  $I^i(\theta,w)$ is given  by \eqref{eq:55}.   It is easy to see that $\theta\mapsto I^i(\theta,w)$ is decreasing on  $\ds \left(0,  {2|x^i|}/{(v^i+w)}\right]$ and increasing on $\ds \left[  {3|x^i|}/{(v^i+2w)}, T\right]$. \\ In $\ds \left[ {2|x^i|}/{(v^i+w)},  {3|x^i|}/{(v^i+2w)}\right]$, $I^i(\theta,w)= P(1/\theta)$, where $P$ is the third order polynomial:
  \begin{displaymath}
    P(z)= 6 (x^i)^2 z^3 + 6 x^i(v^i+w) z^2 +2  ((v^i)^2+v^iw+w^2)z.
  \end{displaymath}
  The roots of the second order polynomial $P'(z) =18 (x^i)^2 z^2 +  12 x^i(v^i+w) z +2  ((v^i)^2+v^iw+w^2)$
 are $ \ds \frac { v^i+w \pm \sqrt{v^iw}   }{3 |x^i|}$.
  Hence, $\theta\mapsto I^i(\theta,w)$ is decreasing in $\ds \left [\frac {2|x^i|}{v^i+w},
  \frac {3 |x^i|}{ v^i+w + \sqrt{v^iw}   }\right] $ and increasing in $\ds \left[ \frac {3 |x^i|}{ v^i+w + \sqrt{v^iw}   }, \frac {3|x^i|}{v^i+2w}\right]$.
  \\
  Therefore, the minimizer of $\theta \mapsto I^i(\theta,w)$ on $[0,T)$ is $\theta=  \frac {3 |x^i|}{ v^i+w + \sqrt{v^iw}   }$ and the minimal value is
  \begin{displaymath}
    \begin{array}[c]{ll}
&P\left(  \frac { v^i+w + \sqrt{v^iw}   }  {3 |x^i|}\right )\\
    = &\ds \frac {2 (v^i)^3} {9 |x^i|} \left(1+\sqrt{\frac w {v^i}} + \frac w {v^i}\right)^3 -\frac  {2 (v^i)^3} {3 |x^i|}  \left(1+\sqrt{\frac w {v^i}} + \frac w {v^i}\right)^2 +\frac  {2 (v^i)^3} {3 |x^i|}  \left(1+\sqrt{\frac w {v^i}} + \frac w {v^i}\right) \\
&+O\left((v^i)^2 \frac {w}{|x^i|}\right)
 \\   =&   \ds  \frac {2 (v^i)^3} {9 |x^i|} +O\left((v^i)^2 \frac {w}{|x^i|}\right).
    \end{array}
  \end{displaymath}  
\end{proof}
 The next lemma is the counterpart of Lemma~\ref{sec:optim-contr-probl-4} when  \eqref{eq:52} holds. By contrast with the situations considered so far,  Assumption \ref{sec:setting-notations-1bis} is used.
 \begin{lemma} \label{sec:clos-graph-prop-8}
   Under Assumptions     \ref{sec:optim-contr-probl-2} and 
\ref{sec:setting-notations-1bis}, consider  a sequence $(x^i, v^i)_{i\in \N}$ which tends to $(x,v)=(0,0)$  as $i\to \infty$, and which satisfies  \eqref{eq:52}. Let  a sequence $(\xi^i, \eta^i)_{i\in \N}$ be such that for all $i\in \N$,    $ (\xi^i, \eta^i)\in \Gamma^{\rm {opt}}[x^i, v^i]$. If $(\xi^i, \eta^i)$ tends to $(\xi,\eta)$ uniformly in $[0,T]$, then  $\eta\in W^{1,2} (0,T;\R)$ and 
    $(\xi,\eta)\in \Gamma^{\rm opt} [x,v]$.
\end{lemma}
 \begin{proof}
   The proof is more difficult than that of Lemma 
\ref{sec:optim-contr-probl-4} because we will see that  in general,
the sequence $u(x^i,v^i)$ does not converge to $u(0,0)$ as $i\to \infty$, and that
   $ \ds \int _0^T \left (\frac {d  \eta^i} {dt}(s)\right)^2 ds$ may tend to $+\infty$.
 
 \paragraph{Step 1 :} 
We start by building a particular competitor for the optimal control problem  at $(x^i, v^i)$. It will be used in Steps 2 and 3 below.
Let us set  $\widetilde t _i = {3|x^i|}/ {v^i}$ (observe that $\lim_{i\to \infty} \widetilde t_i=0$
 since $v^i\to 0$ and ${(v^i)^3}/{|x_i|}\ge C>0$).
As in the proof of Lemma \ref{sec:clos-graph-prop-4} with $w=0$, we construct a pair of continuous functions  $(\widetilde \xi^i , \widetilde \eta^i)$ defined on $[0, \widetilde t _i]$ such that $-1\le \widetilde \xi^i \le 0$ and  $\frac {d \widetilde \xi^i}{dt}=  \widetilde \eta^i$, and 
\begin{enumerate}
\item $(\widetilde \xi^i(0) , \widetilde \eta^i(0))=(x^i, v^i)$
\item$ \widetilde \eta^i(\widetilde t_i)=0$
\item  $\ds \frac 1 2   \int_0^{\widetilde t_i}    \left (\frac {d   \widetilde \eta^i} {dt}(s)\right)^2  ds \sim  \frac {2}{9} \frac {(v^i)^3} {|x^i|}$,
\end{enumerate}
(we have also used Lemma \ref{sec:clos-graph-prop-5} with $w=0$  and  Remark \ref{sec:clos-graph-prop-6}).
Observe that $x^i\le  \widetilde\xi^i( \widetilde t_i)\le 0$, hence $\lim_{i\to +\infty}   \widetilde\xi^i( \widetilde t_i)=0$.
Then, using the same arguments as in Lemma~\ref{sec:clos-graph-prop-12}, it is possible to extend continuously $( \widetilde\xi^i,  \widetilde\eta^i)$ to $[0,T]$
 in such a way that
\begin{enumerate}
\item $( \widetilde\xi^i,  \widetilde\eta^i)\in \Gamma[x^i,v^i]$
\item $\ds \lim_{i\to \infty} \int_{\widetilde t_ i }^T \left| \frac {d   \widetilde \eta^i} {dt}(s)- \alpha
(s - \widetilde t_ i)\right| ^2 ds =0$, where $\alpha$ is an optimal control law for trajectories
with initial values $(0, 0)$.
\end{enumerate}

Combining all the information above, we obtain that 
\begin{equation}
  \label{eq:60}
  J\left(\widetilde \xi^i , \widetilde \eta^i, \frac {d   \widetilde \eta^i} {dt}\right)
  =  \frac {2}{9} \frac {(v^i)^3} {|x^i|} + u(0,0)  +o(1).
 \end{equation}

 \paragraph{Step 2 :} 
   Since  $ (\xi^i, \eta^i)\in \Gamma^{\rm {opt}}[x^i, v^i]$, we know that for all $t\in [0,T]$, $\ds \xi^i (t)= x^i + \int _0^t \eta^i(s) ds\le 0$.  We claim that there exists $t _i \in (0,T]$ such that $\eta^i (t_i) \le -{x^i} /T$. Indeed, if it was not the case, then $\xi^i (T)$ would be larger than $ x^i - T  ({x^i}/ T) = 0$, which is not true. Since $\eta^i$ is continuous, we may define $\theta_i$ as the minimal time $t$ such that  $\eta^i (t) \le -{x^i}/ T$, and we see that  $\eta^i (\theta_i) = -{x^i}/ T$.

\medskip

   Step 2 consists of  proving  that
   \begin{equation}
     \label{eq:61}
     \lim_{i\to \infty} \theta_i=0.
   \end{equation}
Suppose by contradiction that there exists $\delta>0$ such that $\theta_i \ge \delta$. 
 We may apply Lemma  \ref{sec:clos-graph-prop-4} with $w=  {|x^i|}/ T$.
Since $v^i\to 0$ and ${(v^i)^3}/{|x^i|}\ge C>0$, we see that $ |x^i|/v^i\to 0$, then  that 
 $\lim_{i\to \infty} \frac  {3 |x^i|}  {v^i +2 |x^i|/T }=0$.
Hence, for $i$ large enough,  $\theta_i\ge \delta > \frac {3 |x^i|}{v^i +2 |x^i|/T }$, and the third line of  \eqref{eq:55} yields
   \begin{equation}
     \label{eq:62}
     \frac 1 2 \int _0^T \left (\frac {d  \eta^i} {dt}(s)\right)^2 ds \ge
     \frac {2T}{9 (T-\delta)} \frac {(\frac {x^i} T + v^i)^3} {|x^i|}=
     \frac {2T}{9 (T-\delta)} \frac {(v^i)^3} {|x^i|} +o(1),
   \end{equation}
   where $o(1)$ is a quantity that tends to zero as $i\to \infty$, (in fact like $(v^i)^2$). 
   \\
   Note that   $\eta^i \ge - {x^i}/ T\ge 0$ in $[0,\theta_i]$ yields that
   $ \xi^i \ge x^i$     in $[0,\theta_i]$. Therefore
   \begin{equation}
     \label{eq:63}
   \lim_{i\to \infty}   \|  \xi^i\|_{L^\infty(0,\theta_i)}   =0.    
   \end{equation}
   Let us construct an admissible trajectory $(\widehat \xi^i,\widehat \eta^i )$ starting from $(x,v)=(0,0)$ as follows:
    \begin{enumerate}
\item For $s\in  [0,\theta_i]$,  $ \widehat\xi^i(s)= Q_{\theta_i, 0,0, \xi^i(\theta_i),  \eta^i(\theta_i) } (s)$ and 
 $ \widehat\eta^i(s)= Q'_{\theta_i, 0,0, \xi^i(\theta_i),  \eta^i(\theta_i) }(s)$, see Definition~\ref{sec:clos-graph-prop-2}   
\item $(\widehat\xi^i(s), \widehat\eta^i(s))=
      (\xi^i(s), \eta^i(s))$ for $s\in [\theta_i, T]$.
    \end{enumerate}
    It is easy to check that, if $s\le \theta_i$, then
    \begin{eqnarray}
      \label{eq:64}
      \widehat\xi^i(s)&=&  \ds \left(\theta_i \eta^i(\theta_i)  -2 \xi^i(\theta_i)\right) \frac {s^3} {\theta_i^3}  -  \left(\theta_i \eta^i(\theta_i)  -3 \xi^i(\theta_i)\right) \frac {s^2} {\theta_i^2} , \\
       \label{eq:65}
      \widehat\eta^i(s)&=& \ds  3 \left(\theta_i \eta^i(\theta_i)  -2 \xi^i(\theta_i)\right) \frac {s^2} {\theta_i^3}  - 2 \left(\theta_i \eta^i(\theta_i)  -3 \xi^i(\theta_i)\right) \frac {s} {\theta_i^2} ,  \\
        \label{eq:66}
      \frac {d\widehat\eta^i}{dt}(s)&=&  \ds  6 \left(\theta_i \eta^i(\theta_i)  -2 \xi^i(\theta_i)\right) \frac {s} {\theta_i^3}  - 2 \left(\theta_i \eta^i(\theta_i)  -3 \xi^i(\theta_i)\right) \frac {1} {\theta_i^2} .
    \end{eqnarray}
Since $\eta^i(\theta_i)  = - {x^i} /T>0$ and $\xi^i (\theta_i)\le 0$, we see that 
$\left(\theta_i \eta^i(\theta_i)  -2 \xi^i(\theta_i)\right)\ge 0$
and that $ \left(\theta_i \eta^i(\theta_i)  -3 \xi^i(\theta_i)\right) \ge 0$.
Hence  for $s\in [0, \theta_i]$, 
$\widehat\xi^i (s)= \ds \left(\theta_i \eta^i(\theta_i)  -2 \xi^i(\theta_i)\right) \left(\frac {s^3} {\theta_i^3}- \frac {s^2} {\theta_i^2}\right) + \xi^i(\theta_i)  \frac {s^2} {\theta_i^2} \le 0$ as the sum of two nonpositive terms. Therefore $(\widehat \xi^i,\widehat \eta^i )\in \Gamma[0,0]$. On the other hand, using \eqref{eq:63} and the fact that $\theta_i \eta^i(\theta_i) =  {\theta_i} |x^i|/T$, then  \eqref{eq:64} and  \eqref{eq:65}, we see that
\begin{equation}
  \label{eq:67}
\ds \lim_{i\to +\infty} \left(\|\widehat \xi^i\|_{L^\infty (0,\theta_i)}+\|\widehat \eta^i ) \|_{L^\infty (0,\theta_i)}\right)=0.    
\end{equation}
 Moreover, since $\theta_i\ge \delta>0$, it is easy to check that
\begin{equation}
  \label{eq:68}
  \lim_{i\to +\infty}   \int _0^{\theta_i} \left (\frac {d  \widehat \eta^i} {dt}(s)\right)^2 ds=0.  
\end{equation}
Since $(\widehat \xi^i,\widehat \eta^i )\in \Gamma[0,0]$,
\begin{displaymath}
  \begin{array}[c]{rl}
  u(0,0)  \le &J(\widehat \xi^i,\widehat \eta^i,0)\\
    =&\ds  \int_0^T \left(\ell(\widehat \xi^i(s),\widehat \eta^i(s) ,s)
+\frac 1 2    \left (\frac {d  \widehat \eta^i} {dt}(s)\right)^2 \right) ds+g(\widehat \xi^i(T), \widehat \eta^i(T))\\
    =& \ds  \int_0^{\theta_i} \left(\ell(\widehat \xi^i(s),\widehat \eta^i(s) ,s)
 +\frac 1 2    \left (\frac {d  \widehat \eta^i} {dt}(s)\right)^2 \right) ds \\
                &+ \ds\int_{\theta_i}^T  \left(\ell( \xi^i(s), \eta^i(s) ,s)
 +\frac 1 2    \left (\frac {d   \eta^i} {dt}(s)\right)^2 \right) ds +g(\xi^i(T),  \eta^i(T))\\
    =&  \ds u(x^i, v^i)+ \int_0^{\theta_i} \left(\ell(\widehat \xi^i(s),\widehat \eta^i(s) ,s)
+\frac 1 2    \left (\frac {d  \widehat \eta^i} {dt}(s)\right)^2 \right) ds\\
                &\ds - \int_0^{\theta_i} \left(\ell( \xi^i(s), \eta^i(s) ,s)
 +\frac 1 2    \left (\frac {d   \eta^i} {dt}(s)\right)^2 \right) ds.
     \end{array}
\end{displaymath}
Therefore
\begin{equation}  \label{eq:69}
   \begin{array}[c]{rl}
     u(x^i, v^i) \ge &\ds  u(0,0)  + \frac 1 2 \int_0^{\theta_i}    \left (\frac {d   \eta^i} {dt}(s)\right)^2  ds \\
                        &\ds + \int_0^{\theta_i} \left(\ell( \xi^i(s), \eta^i(s) ,s) -\ell( \widehat\xi^i(s), \widehat\eta^i(s) ,s)\right) ds                       
\ds  -\frac 1 2   \int_0^{\theta_i}    \left (\frac {d   \widehat \eta^i} {dt}(s)\right)^2  ds.
   \end{array}            
\end{equation}
Let us address  the terms in the right hand side of  \eqref{eq:69} separately.
\\
Thanks to the  continuity of $\ell$,  \eqref{eq:63},  \eqref{eq:67} and Assumption  
 \ref{sec:setting-notations-1bis},
we see that
  \begin{equation}
    \label{eq:70}
    \begin{array}[c]{rl} 
    &\ds \liminf_{i\to \infty}  \int_0^{\theta_i} \left(\ell( \xi^i(s), \eta^i(s) ,s) -\ell( \widehat\xi^i(s), \widehat\eta^i(s) ,s)\right) ds \\ =&\ds 
    \liminf_{i\to \infty}  \int_0^{\theta_i} \left(\ell( 0, \eta^i(s) ,s)                               -\ell( 0, 0 ,s)\right) ds   \ge 0.
    \end{array}
  \end{equation}
 Combining \eqref{eq:70},  
\eqref{eq:68} and  \eqref{eq:62}, we obtain that
 \begin{equation}
   \label{eq:72}
 u(x^i, v^i)  \ge  \frac {2T}{9 (T-\delta)} \frac {(v^i)^3} {|x^i|} + u(0,0)+ o(1),
\end{equation}
where $o(1)$ is quantity that tends to $0$ as $i\to \infty$.\\
 But for $(\widetilde \xi^i , \widetilde \eta^i)$ constructed in Step 1,   $ J\left(\widetilde \xi^i , \widetilde \eta^i, \frac {d   \widetilde \eta^i} {dt}\right) \ge u(x^i, v^i)$. This fact and   \eqref{eq:60} lead to a contradiction with  \eqref{eq:72}.
 We have proved   \eqref{eq:61}.

 \paragraph{Step 3 :}
 Since $\lim_{i\to \infty} \theta_i =0$ and $(\xi^i, \eta^i)$ converges uniformly to $(\xi, \eta)$, we see that
 \begin{displaymath}
   \int_0^{\theta_i}
\ell(\xi^i(s), \eta^i(s) ,s)
 ds =o(1).
 \end{displaymath}
 Hence
 \begin{equation}
    \label{eq:73}
   \begin{split}
     u(x^i, v^i)=&   J\left(\xi^i, \eta^i , \frac {d  \eta^i} {dt}\right)\\
     =& \frac 1 2   \int_0^{\theta_i}    \left (\frac {d   \eta^i} {dt}(s)\right)^2  ds \\
&+ \int_{\theta_i}^T 
\ell(\xi^i(s), \eta^i(s) ,s)
 ds +
\frac 1 2   \int_{\theta_i}^T    \left (\frac {d   \eta^i} {dt}(s)\right)^2  ds + g(\xi^i(T), \eta^i(T)) +o(1).     \end{split}
\end{equation}
 On the other hand, we have seen above  that  \eqref{eq:60} implies that
 \begin{equation}
    \label{eq:74}
  u(x^i, v^i)\le  \frac {2}{9} \frac {(v^i)^3} {|x^i|} + u(0,0)  +o(1).
\end{equation}
From Lemma \ref{sec:clos-graph-prop-5}, we know that
 \begin{equation}
    \label{eq:75}
 \frac 1 2   \int_0^{\theta_i}    \left (\frac {d   \eta^i} {dt}(s)\right)^2  ds\ge  \frac {2}{9} \frac {(v^i)^3} {|x^i|}   -o(1).
  \end{equation}
  Combining  \eqref{eq:73}, \eqref{eq:74} and  \eqref{eq:75} yields that
  \begin{equation}
    \label{eq:76}
    \begin{split}
\int_{\theta_i}^T 
\ell(\xi^i(s), \eta^i(s) ,s)
 ds +
\frac 1 2   \int_{\theta_i}^T    \left (\frac {d   \eta^i} {dt}(s)\right)^2 ds 
+ g(\xi^i(T), \eta^i(T))
 \le 
u(0,0) +o(1).      
    \end{split}
\end{equation}
Since $(\xi^i, \eta^i)$ converges uniformly to $(\xi, \eta)$,
   \eqref{eq:76} implies that $\left(\one_{(\theta_i,T)} \frac {d   \eta^i} {dt}\right)_{i\in \N}$ is a bounded sequence in $L^2 (0,T)$. Hence there exists $\phi\in L^2 (0,T)$ such that, after the extraction of subsequence, $\one_{(\theta_i,T)} \frac {d   \eta^i} {dt} \rightharpoonup \phi$ in $ L^2 (0,T)$ weak. By testing with compactly supported functions in $(0,T)$, it is clear that $\phi=  \frac {d   \eta} {dt}$. Hence, the whole sequence  $\left(\one_{(\theta_i,T)} \frac {d   \eta^i} {dt}\right)_{i\in \N}$ converges in $ L^2 (0,T)$ weak to $ \frac {d   \eta} {dt}\in L^2(0, T)$. Moreover, the weak convergence in $ L^2 (0,T)$ implies that
\begin{displaymath}
  \int_{0}^T \left(\frac {d   \eta} {dt}(s)\right)^2  ds \le \liminf_{i\to +\infty}
 \int_{0}^T \left(\one_{(\theta_i,T)}  \frac {d   \eta^i} {dt}(s)\right)^2 ds.
\end{displaymath}
This and  \eqref{eq:76} imply that
\begin{equation}
  \label{eq:77}
   \int_{0}^T 
\ell(\xi(s), \eta(s) ,s)
 ds +
\frac 1 2   \int_{0}^T    \left (\frac {d   \eta} {dt}(s)\right)^2ds 
+ g(\xi(T), \eta(T))  \le u(0,0) .
\end{equation}
Hence, $(\xi, \eta)\in \Gamma^{\rm opt} [0,0]$ and the above inequality is in fact an identity. The proof is achieved.
 \end{proof}


 \begin{proof} [Proof of Theorem \ref{sec:clos-graph-prop-1}]
   Consider  $(x,v)\in\Xi^{\rm{ad}}$ and a sequence $(x^i, v^i)_{i\in \N}$ such that for all $i\in \N$,   $(x^i, v^i)\in\Xi^{\rm{ad}}$ and $(x^i, v^i)\to (x,v)$ as $i\to \infty$.
   Consider a sequence     $(\xi^i,\eta^i)_{i\in \N}$ such that for all $i\in \N$, $(\xi^i, \eta^i) \in \Gamma^{\rm{opt}}[x^i, v^i]$ and that  $(\xi^i,\eta^i)$ tends to $(\xi,\eta)$
   uniformly. Possibly after the extraction of a subsequence, we can always assume that
   either one among the five conditions in Lemma \ref{sec:optim-contr-probl-4}  or one among the two symmetrical conditions   \eqref{eq:52}-\eqref{eq:5019bis} holds.
   Then the conclusion follows from Corollary \ref{sec:clos-graph-prop-1} in the former case or from Lemma  \ref{sec:clos-graph-prop-8} in the latter case.
 \end{proof}
\begin{remark}
 For costs of the form
 \begin{equation}
J(\xi,\eta, \alpha)
=\int_0^T \left(\ell(\xi(s),\eta(s) ,s) +\frac 1  p   |\alpha|^p(s) \right) ds+g(\xi(T), \eta(T)),
\end{equation}
with  $1< p\not=2$, it is not possible to reproduce the explicit calculations of Lemmas  \ref{sec:clos-graph-prop-4} and  \ref{sec:clos-graph-prop-5}, which are crucial steps for Lemma  \ref{sec:clos-graph-prop-8} and finally for Theorem \ref{sec:clos-graph-prop-1}.
\end{remark}

 \subsection{Bounds related to optimal trajectories}
 \label{sec:bounds}

 \begin{proposition}\label{sec:variant-which-domain}
   For positive numbers $r$ and $C$,  let us set
   \begin{eqnarray}
     \label{eq:78}
     \Theta_r&=&\left\{(x,v)\in  \Xi:     -r (x+1)  \le v^3 \le r |x| \right\},\\
     \label{eq:79}
     K_C&=& \{  (x,v)\in  \Xi: \; |v|\le C\},   
     \\
     \label{eq:80}
     \Gamma_C &=& \left\{  (\xi,\eta)\in \Gamma: \left|
         \begin{array}[c]{l}
           (\xi(t), \eta(t))\in K_C, \quad \forall t\in [0,T],\\
        \left \| \frac {d\eta}{dt} \right\|_{L^2(0,T;\R)}\le C.
      \end{array}
    \right.  \right\}.
\end{eqnarray}
Under Assumption~\ref{sec:optim-contr-probl-2}, for all $r>0$, there exists a positive number $C= C(r, M)$
($M$ is defined in Assumption~\ref{sec:optim-contr-probl-2})  such that
if  $(x,v)\in \Theta_r$, then $\Gamma^{\rm opt} [x,v]\subset  \Gamma_C$. Moreover, as $r\to +\infty$,  $C(r,M)= O(\sqrt{r})$.

\end{proposition}
\begin{proof}
 A possible proof consists of building a suitable map $j$ from $\Theta_r$ to $\Gamma$. We make out different cases:
  \begin{description}
  \item[Case 1: $0\le v\le  {-3x}/T$:] let $j(x,v)= (\widetilde \xi,\widetilde \eta)\in \Gamma[x,v]$ be defined by
     \begin{displaymath}
\left\{ \begin{array}[c]{rcllrcl}
\widetilde \eta(t)&=& v\left(1-\frac {3t}{2T} \right)\quad &\hbox{and} \quad &\widetilde \xi(t)= &x+v\left(t-\frac {3t^2}{4T}\right),\quad  &\hbox{if } \quad 0\le t\le \frac {2T} 3 ,
\\
\widetilde \eta(t)&=0&\quad&\hbox{and} \quad &\widetilde \xi(t)= &x+ \frac {vT} 3, \quad  &\hbox{if } \quad  \frac {2T} 3 \le t\le T.
\end{array}\right.
\end{displaymath}
It is easy to check that there exists a constant $\widetilde C=\widetilde C(r,M)$ such that 
\begin{equation}
\label{eq:81}
  \|\widetilde\eta\|_{L^\infty(0,T;\R)} \le \widetilde C;\quad \quad \quad \left \| \frac {d\widetilde \eta}{dt} \right\|_{L^2(0,T;\R)}\le \widetilde C.
\end{equation}
  \item[Case 2: $ -3x/T< v\le  r |x|^{\frac 1 3}$:]  in this case, we choose $j(x,v)=(\widetilde \xi,\widetilde \eta)\in \Gamma[x,v]$ where $\widetilde \xi'=\widetilde \eta$ and $\widetilde \eta$ is the solution of the linear complementarity problem \eqref{eq:56} with $\theta=T$.
Here again $(\widetilde \xi,\widetilde \eta)$ satisfies (\ref{eq:81}) for some constant 
$\widetilde C=\widetilde C(r,M)$. From Lemma~\ref{sec:clos-graph-prop-4}, we see that as $r\to +\infty$,
  $\widetilde C=O(\sqrt{r})$.
  \item[Case 3: $- {3(1+x)}/T\le v\le 0$:] the situation is symmetric to Case 1, and $j(x,v)$ is given by the same formula.
 \item[Case 4: $ -r   (x+1)^{\frac 1 3}\le v< -{3(1+x)}/T$:]
 the situation is symmetric to Case 2, and $j(x,v)$ is constructed in the symmetric way as in Case 2.
  \end{description}
Then, using $j(x,v)$ as a competitor for the optimal control problem leads to the desired result with a constant $C$ that depends only on $r$ and $M$ and that can always be taken larger than $\widetilde C$.\\
Note that $j$ is piecewise continuous from $\Theta_r$ to $\Gamma$.     Note also that the construction of $j$ is independent  of $\ell$ and $g$.
\end{proof}
\begin{remark}
  Note that the sets $\Theta_r$ form an increasing family of compact subsets of $\Xi^{\rm{ad}}$ and that 
  \begin{equation}
    \label{eq:82}
    \bigcup_{r\ge 0} \Theta_r= \Xi^{\rm{ad}}.
  \end{equation}
\end{remark}
 \subsection{Mean field games with state constraints}
\label{sec:mean-field-games-1}
In the example considered here, we take  $\Xi=[-1,0]\times \R$.  Let $\cP(\Xi)$ be the set of probability measures on $\Xi$.
\\
Let $F, G: \cP(\Xi)\to C_b^0(\Xi;\R)$ be  bounded and continuous maps (the continuity is  with respect to the narrow convergence  in $ \cP(\Xi)$) and $\ell$ be a continuous and  bounded from below function defined on $
 \Xi \times [0,T]$.
     Set
     \begin{equation}
       \label{eq:100}
M=\max\left(  \sup_{(x,v,t)\in \Xi \times [0,T]} \ell_- (x,v,t)+ \sup_{m\in \cP(\Xi)}  \|F[m]\|_{L^\infty(\Xi)} , \sup_{m\in \cP(\Xi)} \|G[m]\|_{L^\infty(\Xi)}\right) .  
     \end{equation}
\begin{assumption}   \label{sec:MFG-contr-accel2}
We  assume that for all $t\in [0,T]$, $m\in  \cP(\Xi)$ and $v\ge 0$, $\ell(0,v,t)+F[m] (0,v)\ge \ell(0,0,t)+F[m] (0,0)$ and  $\ell(-1,-v,t)+ F[m](-1,-v)\ge \ell(-1,0,t)+ F[m] (-1,0)$.  
\end{assumption}
Using similar notations as in paragraph \ref{sec:MFG-contr-accel1},
 we consider the cost given by (\ref{costMFG}).
With $M$ in (\ref{eq:37}), note that Proposition \ref{sec:variant-which-domain} can be applied to $J^{\mu}$ defined in (\ref{costMFG}) with  constants $C(r,M)$ uniform in $\mu$.
\begin{lemma}
\label{sec:mean-field-games}
Let $r$ be a positive number.
Under the assumptions made above on $\ell$, $F$ and $G$ (including Assumption \ref{sec:MFG-contr-accel2}), let  $C=C(r,M)$ be the constant appearing in Proposition~\ref{sec:variant-which-domain}. For any probability measure  $m_0$ on $\Xi$ supported in $\Theta_r$ defined in (\ref{eq:78}), there exists a 
constrained mean field game equilibrium  associated with the initial distribution $m_0$, i.e.  a probability measure $\mu \in \cP_{m_0}(\Gamma_C)$ such that  (\ref{eq:43}) holds.
\end{lemma}
\begin{proof}
  The proof is similar to that of  Theorem \ref{sec:setting-notation-6}. We skip it.
\end{proof}
\begin{remark}\label{sec:mean-field-games-4}
Compared to  Theorem \ref{sec:setting-notation-6},  the restrictions made on the support of  $m_0$ are weaker in Lemma~\ref{sec:mean-field-games}, but the latter requires the  additional  Assumption \ref{sec:MFG-contr-accel2}.\\ In Theorem \ref{sec:mean-field-games-5} below, we get rid of the assumptions on the support of $m_0$.
\end{remark}
\begin{theorem}
\label{sec:mean-field-games-5}
Let $m_0$ be a probability measure on $\Xi$ such that 
\begin{equation}
  \label{eq:83}
m_0(\Xi \setminus \Xi^{ \rm{ad}} )=0.
\end{equation}
Under the assumptions made above on $\ell$, $F$ and $G$ (including Assumption \ref{sec:MFG-contr-accel2}), there exists a 
constrained mean field game equilibrium  associated with the initial distribution $m_0$, i.e.  a probability measure $\mu \in \cP_{m_0}(\Gamma)$ such that  (\ref{eq:43}) holds.
\end{theorem}
\begin{proof}
From (\ref{eq:82}) and (\ref{eq:83}),  there exists $n_0>0$ such that $m_0(\Theta_{n})>0$ for $n>n_0$. 
For $n>n_0$, we set $m_{0,n}= \frac 1 {m_0 ( \Theta_n)} m_0|_{\Theta_n}$. With a slight abuse of notation, let $m_{0,n}$ also denotes the probability on $\Xi$ obtained by extending $m_{0,n}$ by $0$ outside $\Theta_n$, i.e.
$m_{0,n}(B)= \frac 1 {m_0 ( \Theta_n)} m_0(B\cap \Theta_n)$, for any measurable subset $B$ of $\Xi$. Let $\mu_n\in \cP_{m_{0,n}}\left( \Gamma_{C(n,M)}\right)$ be   a constrained mean field game equilibrium  associated with the initial distribution $m_{0,n}$, the existence of which comes from Lemma~\ref{sec:mean-field-games}. With a similar abuse of notations as above,  let $\mu_{n}$ also denote the probability on $\Gamma$ obtained by extending $\mu$ by $0$ outside $ \Gamma_{C(n,M)}$.

\medskip

We claim that $\{\mu_n, n>n_0\}$ is tight in $\cP(\Gamma)$, i.e. that for each $\epsilon>0$, there exists a compact $K_\epsilon \subset \Gamma$ such that
\begin{equation}
  \label{eq:84}
\mu_n (\Gamma\setminus K_\epsilon)<\epsilon,\quad \quad \hbox{for each }n>n_0.
\end{equation}
From the increasing character of the sequence $\Theta_n$,  (\ref{eq:82}) and (\ref{eq:83}), we observe that for each $\epsilon>0$, there exists $n_1>0$ such that $m_0(\Theta_{n_1})>1-\epsilon$. Let us prove (\ref{eq:84}) with $K_\epsilon= \Gamma_{C(n_1, M)}$. \\
Since for all $n>n_0$, $\mu_n\in \cP_{m_{0,n}}(\Gamma)$ is a MFG equilibrium, we see that for all measurable $B\subset \Xi$,
\begin{displaymath}
  \begin{split}
  m_{0,n}(B)
=
\mu_n\left\{   (\xi,\eta)\in\hbox{supp}(\mu_n)\;:\; (\xi(0),\eta(0))\in B  \right\}
\le
\mu_n\left( \bigcup_{(x,v)\in B }  \Gamma^{{\rm{opt}},\mu_n} [x,v] \right)  .
  \end{split}
\end{displaymath}
Taking $B=\Theta_{n_1}$ and using Proposition~\ref{sec:variant-which-domain}, we see that 
\begin{displaymath}
   m_{0,n}\left(\Theta_{n_1} \right)
\le \mu_n\left( \bigcup_{(x,v)\in \Theta_{n_1}  }  \Gamma^{{\rm{opt}},\mu_n} [x,v] \right)  
\le \mu_n\left( \Gamma_{C(n_1,M)}\right),
\end{displaymath}
(note that the constant $C(n_1,M)$ does not depend on $\mu_n$).\\
On the other hand,
\begin{displaymath}
  \begin{array}[c]{rcll}
  m_{0,n}\left(\Theta_{n_1}\right) &\ge & m_0(\Theta_{n_1})> 1 -\epsilon, \quad &\hbox{if } n> n_1, \\
  m_{0,n}\left(\Theta_{n_1}\right) &=&1,    \quad &\hbox{if } n_0< n \le  n_1.
  \end{array}
 \end{displaymath}
In both cases, $\mu_n\left( \Gamma_{C(n_1,M)}\right) \ge 1-\epsilon$ and therefore $\mu_n\left( \Gamma \setminus \Gamma_{C(n_1,M)}\right) \ge 1-\epsilon$, and the claim is proved.

\medskip

Thanks to Prokhorov theorem, possibly after the  extraction of subsequence that we still name $\mu_n$, we deduce that there exists $\mu\in \cP(\Gamma)$ such that $\mu_n$ converges narrowly to $\mu$. 

\medskip

We claim that $\mu$ is a MFG equilibrium related to $m_0$. We already know that $\mu\in \cP(\Gamma)$. There remains to prove that 
\begin{itemize}
\item $\mu \in \cP_{m_0}(\Gamma)$, i.e. that $e_0\sharp \mu=m_0$
\item  $\mu$ satisfies (\ref{eq:43}).
\end{itemize}

The fact that $e_0\sharp \mu=m_0$ stems from Lemma \ref{sec:setting-notation-4} and from the fact that $m_{0,n}$ narrowly converges to $m_0$.

In order to prove (\ref{eq:43}), we recall that from Kuratowski's theorem, see \cite{MR2129498},
\begin{displaymath}
  {\rm{supp}}(\mu)\subset \liminf_{n\to \infty}  {\rm{supp}}(\mu_n),
\end{displaymath}
which means that for all $(\xi, \eta)\in {\rm{supp}}(\mu)$, there exists a sequence  $(\xi_n, \eta_n)\in {\rm{supp}}(\mu_n)$ such that $(\xi_n, \eta_n)\to (\xi, \eta)$ uniformly. As a consequence, setting $(x_n,v_n)= (\xi_n(0), \eta_n(0))$ and $(x,v)= (\xi(0), \eta(0))$, $\lim_{n\to \infty} (x_n,v_n)=(x,v)$ and $(\xi_n, \eta_n)\in \Gamma^{{\rm{opt}},\mu_n}[x_n,v_n]$.  Applying Proposition \ref{sec:mean-field-games-6} below, which is a generalization of Theorem~\ref{sec:optim-contr-probl-3},  we may pass to the limit and conclude that $(\xi, \eta)\in \Gamma^{{\rm{opt}},\mu}[x,v]$, which achieves the proof.
\end{proof}

\begin{proposition}
  \label{sec:mean-field-games-6}
Under the assumptions made above on $\ell$, $F$ and $G$ (including Assumption \ref{sec:MFG-contr-accel2}), consider a sequence $(\mu ^i)_{i\in \N}$, $\mu ^i\in \cP(\Gamma)$, such that $\mu^ i$ converges narrowly  to $\mu\in \cP(\Gamma)$.  Consider a sequence $(\xi^i,\eta^i)_{i\in \N} $, $(\xi^i,\eta^i)\in \Gamma$, such that
\begin{enumerate}
\item $(\xi^i,\eta^i)\in \Gamma^{{\rm{opt}}, \mu^ i} [x^i, v^i]$, where $(x^i, v^i)=(\xi^i(0),\eta^i(0))$
\item  $(\xi^i,\eta^i)$ tends to $ (\xi,\eta)\in \Gamma[x,v]$ uniformly, where $(x,v)= \lim_{i\to \infty}(x^i, v^i)$. 
\end{enumerate}
Then $(\xi,\eta)\in \Gamma^{{\rm{opt}}, \mu} [x, v]$.
\end{proposition}
\begin{proof}
We skip the proof, because it follows the same lines as that of Theorem ~\ref{sec:optim-contr-probl-3},
(see paragraph~\ref{sec:optim-contr-probl}). In particular, it includes  an adaptation of 
Lemma~\ref{sec:clos-graph-prop-8}. The necessary modifications are obvious.
\end{proof}

\section{State constrained optimal control problems and mean field games in a convex polygonal domain of $\R^2$}\label{sec:contr-probl-polyh}

Let  $\Omega$ be a bounded and convex domain of $\R^2$  with a polygonal boundary $\partial \Omega$.
 For $x\in \overline \Omega$,  the tangent cone to $\Omega$ at $x$ is defined by 
 \begin{displaymath}
   T_\Omega(x)=\left\{ v\in \R^2: \;  x+tv \in \overline \Omega, \hbox{ for $t>0$ small enough}\right \}.
 \end{displaymath}
Note that $T_\Omega(x)=\R^2$ if $x\in \Omega$.
A vector $v\in \R^2$ points outward $\Omega$ at $x\in \partial \Omega$ if $v\notin T_\Omega(x)$.\\
Let $(\nu_i)_{0\le i< N}$ be the vertices of $\partial \Omega$, labeled in such a way that $\partial \Omega= \bigcup_{i=0}^{N-1}   \gamma_i $, where $\gamma_i= [\nu_i, \nu_{i+1}]$ and $\nu_{N}=\nu_0$. We may assume that three successive vertices are not aligned. We are going to use the notation $(\nu_i,\nu_{i+1})$ for the open straight line segment between  $\nu_i$ and $\nu_{i+1}$.
 For $i\in \{0,\dots, N-1\}$, let $n_i$ be the unitary  normal vector to $\gamma_i$ pointing outward  $\Omega$.
It is easy to see that $ T_\Omega(\nu_i)=\{x\in \R^2 \;:\;  n_i\cdot x \le 0  \hbox{ and }  n_{i-1}\cdot x \le 0  \}$,
setting $n_{-1}=n_{N-1}$.
 Since $\Omega$ is convex, $\overline \Omega$ coincides locally near $\nu_i$ with $\nu_i+  T_\Omega(\nu_i)$.

The optimal control problem is set exactly as in Section~\ref{sec:state-constr-optim-1}: it consists of minimizing $J(\xi, \eta, \eta')$ given by (\ref{eq:2}) on the dynamics given by (\ref{eq:1}) and staying in $\Xi=\overline \Omega \times \R^2$. The set $\Xi^{\rm{ad}}$ is defined by (\ref{eq:6}).
  \subsection{Closed graph properties}
\label{sec:clos-graph-prop-13}
 The closed graph result given in Proposition \ref{sec:clos-graph-prop-14} below is similar to that contained in Proposition~\ref{sec:contr-probl-polyg-1}, but special conditions are needed near the vertices of $\partial \Omega$:
\begin{proposition}\label{sec:clos-graph-prop-14}
  Consider a closed subset $\Theta$ of $\Xi^{\rm ad}$.
Assume that  all sequence  $(x^i, v^i)_{i\in \N}$  such that for all $i\in \N$, $ (x^i, v^i)\in \Theta$ and $\lim_{i\to +\infty} (x^i, v^i)=(x, v)\in \Theta$, the following holds:
  \begin{enumerate}
  \item If  $x\in (\nu_j,\nu_{j+1})$ for some $j\in \{0,\dots, N-1\}$
(recall that $\nu_{N}=\nu_0$),  then 
\begin{equation}
  \label{eq:85}
(v^i \cdot n_j)_+^3 = o\Bigl( (x-x^i)\cdot n_j\Bigr) ;
\end{equation}
\item if $x=\nu_j$  for some $j\in \{0,\dots, N-1\}$ and $v\not =0$, then 
 \begin{equation}
   \label{eq:86}
(v^i \cdot n_k)_+^3 = o\Bigl( (x-x^i)\cdot n_k\Bigr), \quad \hbox{for } k=j-1, j, 
 \end{equation}
recalling that $n_{-1}= n_{N-1}$;
\item  If $x=\nu_j$  for some $j\in \{0,\dots, N-1\}$ and $v =0$, then
  \begin{equation}
\label{eq:87}
(v^i \cdot n_k)_+  \left( |x-x^i|^{\frac 2 3} + |v^i|^2\right) = o\Bigl( (x-x^i)\cdot n_k\Bigr), \quad \hbox{for } k=j-1, j;
 \end{equation}
  \end{enumerate}
then  the graph of the multivalued map $ \Gamma^{\rm opt}: \; \Theta \rightarrow \Gamma$, $(x,v) \mapsto \Gamma^{\rm opt}[x,v] $,
is closed in the sense given in Proposition ~\ref{sec:contr-probl-polyg-1}.
\end{proposition}

\begin{remark}
  \label{sec:clos-graph-prop-10}
It is easy to find  sets $\Theta$ which fulfill the assumption of Proposition
\ref{sec:clos-graph-prop-14}, for example:
\begin{enumerate}
\item   a compact subset $\Theta$ of $\R^2\times \R^2$ contained in  $\Omega \times \R^2$ ;
\item  for $j=0,\cdots, N-1$, let $D_j$ be the straight line containing $\nu_j$ and $\nu_{j+1}$; for $\rho>1$, set 
  \begin{displaymath}
    S_j=\left\{(x,v)\in \overline \Omega \times \R^2:  v\cdot n_j \le \hbox{dist}(x, D_j)^\frac \rho 3 \right \}.
  \end{displaymath}
If $\Theta$ is  a closed subset $\R^2\times \R^2$ contained 
in  $\left(\Xi^{\rm ad} \cap \bigcap_{j=0}^{N-1} S_j\right) \setminus  \left(\bigcup_{j=0}^{N-1} \{\nu_j\}\times \R^2\right)$, then $\Theta$ fulfills the assumption   of Proposition~\ref{sec:clos-graph-prop-14}.
\item For brevity, we do not supply  examples of sets $\Theta$ which intersect $ \bigcup_{j=0}^{N-1} \{\nu_j\}\times \R^2$, although it is possible to do so. 
\end{enumerate}
\end{remark}
The proof of Proposition \ref{sec:clos-graph-prop-14} relies on  the counterpart of Lemma~\ref{sec:clos-graph-prop-7} whose proof is exactly the same,
and on   Lemma \ref{sec:contr-probl-polyg} below which is the counterpart of Lemma~\ref{sec:clos-graph-prop-12}:
\begin{lemma}
  \label{sec:contr-probl-polyg}
  Consider  $(x,v)\in\Xi^{\rm{ad}}$, $(\xi,\eta)\in \Gamma[x,v]$ such that $\eta\in W^{1,2} (0,T;\R^2)$ and a sequence $(x^i, v^i)_{i\in \N}$ such that for all $i\in \N$,   $(x^i, v^i)\in\Xi^{\rm{ad}}$ 
  and $\ds \lim_{i\to \infty} (x^i, v^i)= (x,v)$.
  \\
  Assume that one among the following seven conditions is true:
  \begin{enumerate}
  \item $x\in \Omega$
  \item $x\in (\nu_j, \nu_{j+1})$,  $v\cdot n_j < 0$  (hence $v^i\cdot n_j< 0$ for $i$ large enough)
  \item $x\in (\nu_j, \nu_{j+1})$, $v\cdot n_j = 0$ and one among  the following properties is true:
    \begin{enumerate}
    \item   $v^i\cdot n_j\le 0$ for $i$ large enough
    \item  for $i$ large enough, $v^i\cdot n_j> 0$  (hence $(x-x^i)\cdot n_j>0$) and
      \begin{equation}
        \label{eq:88}
        \lim_{i\to \infty}  
         \frac { (v^i\cdot n_j)^3}{(x-x^i)\cdot n_j}=0
      \end{equation}
    \end{enumerate}
  \item $x=\nu_j$, $v\cdot n_j< 0$ and $v\cdot n_{j-1}< 0$ (setting $n_{-1}=n_{N-1}$)
  \item $x=\nu_j$,   $v\cdot n_j= 0$ and $v\cdot n_{j-1}< 0$  and one among the two following properties holds
    \begin{enumerate}
    \item $v^i\cdot n_j\le 0$ for $i$ large enough
    \item   $v^i\cdot n_j>0$  and  (\ref{eq:88}) holds for  $i$ large enough.
    \end{enumerate}
  \item   $x=\nu_j$,   $v\cdot n_{j-1}= 0$ and $v\cdot n_{j}< 0$  and one among the two following properties holds
    \begin{enumerate}
    \item $v^i\cdot n_{j-1}\le 0$ for $i$ large enough
    \item   $v^i\cdot n_{j-1}>0$  and  (\ref{eq:88}) holds (replacing $j$ by $j-1$) for  $i$ large enough
    \end{enumerate}
  \item $x=\nu_j$,   $v=0$  and 
    \begin{equation}
      \label{eq:89}
      \lim_{i\to \infty}  \max_{k\in \{j-1,j\}}
      \frac { (v^i\cdot n_k)_+} {|(x^i -x)\cdot n_k |}   \left(|x-x^i|^{\frac 2 3 }    +|v^i|^{2}\right)    = 0,
    \end{equation}
with the convention that $    \frac { (v^i\cdot n_k)_+} {|(x^i -x)\cdot n_k |} =0$ if $(v^i\cdot n_k)_+=0$.
  \end{enumerate}
 Then  there exists a sequence  $(\xi^i,\eta^i)_{i\in \N}$ such that  $(\xi^i,\eta^i)\in \Gamma [x^i, v^i]$, $\eta^i\in W^{1,2} (0,T;\R^2)$, and  $(\xi^i,\eta^i)$ tends to $(\xi,\eta)$ 
  in $W^{2,2}(0,T;\R^2) \times  W^{1,2}(0,T;\R^2)$, hence uniformly in $[0,T]$.
\end{lemma}

\begin{proof}[Proof of Lemma \ref{sec:contr-probl-polyg}]
Since many arguments are almost identical to those contained in the proof of Lemma~\ref{sec:clos-graph-prop-12}, we just sketch the proof and put  the stress on cases 4-7 which have no counterparts in Lemma~\ref{sec:clos-graph-prop-12}. As for Lemma~\ref{sec:clos-graph-prop-12},  each of the seven conditions  mentioned in the statement  makes it possible to explicitly  construct  families of admissible trajectories fulfilling all the desired properties.
  Since the construction is different in each of the seven cases, we discuss each case separately:
  \begin{enumerate}
 \item If $x\in \Omega$, then the construction is exactly the same as in the first case in the proof of  Lemma~\ref{sec:clos-graph-prop-12}. We do not repeat the argument.
  \item $x\in (\nu_j, \nu_{j+1})$ and  $v\cdot n_j<0$, hence
 for   $i$ large enough,   $v^i\cdot n_j<0$.  We can always assume that the latter property holds for all $i$.
Since $\eta\in W^{1,2}(0,T)$, there exists $\overline{t}\in (0,T)$ such that for all
$s\in [0,\overline t]$, $ \frac 3 2 v\cdot n_j \le \eta(s)\cdot n_j \le  \frac 1 2 v\cdot n_j$ and $ \frac {3s} 2 v\cdot n_j  \le (\xi(s)-x)\cdot n_j \le  \frac s 2 v\cdot n_j$.
For $t_i\in [0,\overline t]$, we set
\begin{equation}
  \label{eq:92}
  \xi^i (s)= \left\{
    \begin{array}[c]{rcl}
      Q_{t_i,x^i, v^i , \xi(t_i),\eta(t_i)}(s),\quad & \hbox{ if } \quad s \in [0,t_i],\\
      \xi(s),\quad & \hbox{ if } \quad s\in [ t_i,T],
    \end{array}
\right.
\end{equation}
and the remaining arguments are very close to those in case 2. in the proof of  Lemma~\ref{sec:clos-graph-prop-12}. We skip the details.
\item 
 \begin{enumerate}
 \item $x\in (\nu_j, \nu_{j+1})$, $v\cdot n_j = 0$ and  $v^i\cdot n_j\le 0$ at least for $i$ large enough.
We may assume that  $v^i\cdot n_j\le 0$ for all $i$.  Given $t_{i,1}\in (0,T)$, we define $(y^i, w^i)$ as follows:
 \begin{equation}
   \label{eq:93}
   \begin{array}[c]{rcl}
   y^i&=& (x\cdot n_j) n_j + ( (x^i+ v^i t_{i,1})\cdot n_j^\perp) n_j^\perp, \\     
   w^i&=& (v\cdot n_j) n_j +  ( v^i\cdot n_j^\perp) n_j^\perp=( v^i\cdot n_j^\perp) n_j^\perp ,     
   \end{array}
 \end{equation}
and, for $t_{i,1}\le t_{i,2}<T$, set 
\begin{equation}
  \label{eq:94}
  \xi^i(s)=\left\{
    \begin{array}[c]{rcl}
      Q_{t_{i,1},x^i, v^i , y^i,w^i}(s) ,\quad & \hbox{ if } &\quad s\le t_{i,1},\\
      Q_{t_{i,2}-t_{i,1}, y^i - x, w^i-v, 0,0 }(s-t_{i,1})+ \xi(s-t_{i,1}), \quad &\hbox{ if }& t_{i,1}\le s\le t_{i,2},\\
      \xi(s-t_{i,1}), \quad &\hbox{ if } &t_{i,2}\le s\le T.
    \end{array}
  \right.
\end{equation}
We argue as in case 3.(a) in the proof of  Lemma~\ref{sec:clos-graph-prop-12}. An important observation is that
\begin{equation}
  \label{eq:95}
   (\xi^i(s) -x)\cdot n_j= \left(  (x^i-x)\cdot n_j \left(1+2 \frac {s} {t_{i,1}} \right)  + sv^i\cdot n_j   \right)
  \left(1-\frac {s} {t_{i,1}}\right)^2
\end{equation}
is non positive for $s\in [0,t_{i,1}]$.
 We skip the other details.
\item $x\in (\nu_j, \nu_{j+1})$, $v\cdot n_j = 0$,  $v^i\cdot n_j > 0$ for all  $i$ (or for $i$ large enough), and (\ref{eq:88}) holds. This case is the counterpart of case 3.(b) in the proof of Lemma~\ref{sec:clos-graph-prop-12}. The trajectory $\xi^i$  is constructed as in (\ref{eq:94}), but  a further restriction on $t_{i,1}$ is needed in order to guarantee that the trajectory is admissible, namely that 
\begin{displaymath}
  t_{i,1} \le    \frac {3|(x^i -x)\cdot n_j |} {v^i\cdot n_j}.
\end{displaymath}
This condition should be supplemented with the other two conditions as in  3.(a), namely that
 \begin{eqnarray}
    \label{eq:96}
\lim_{i\to \infty} t_{i,1}=0,
\\ \label{eq:97}
\lim_{i\to \infty} \frac {  |(x-x^i)\cdot n_j|^2 }{t_{i,1}^3} + \frac {|(v-v^i)\cdot n_j|^2 }{t_{i,1}}=0.
  \end{eqnarray}
If (\ref{eq:88}) holds, then it is possible to choose such a sequence $t_{i,1}$. 
The remaining part of the proof is as in case 3.(a).
\end{enumerate}
\item  $x= \nu_j$, $v\cdot n_j<0$ and $v\cdot n_{j-1}<0$. Since $\eta\in W^{1,2}(0,T)$,
 there exists $\overline{t}\in (0,T)$ such that for all
$s\in [0,\overline t]$,  $ \frac 3 2 v\cdot n_k \le \eta(s)\cdot n_k \le  \frac 1 2 v\cdot n_k$ and $ \frac {3s} 2 v\cdot n_k \le (\xi(s)-x)\cdot n_k \le  \frac s 2 v\cdot n_k$,
for  $k=j-1,j$. For $t_i\in [0,\overline t]$, we choose $ \xi^i$ as in (\ref{eq:92}) and the desired result is obtained as in case 2.
\item $x=\nu_j$, $v\cdot n_j=0$ and $v\cdot n_{j-1}<0$. We make out two subcases:
\begin{enumerate}
\item $v^i \cdot n_j\le 0$ at least for $i$ large enough: the trajectory is constructed as in (\ref{eq:94}),
with three different stages corresponding respectively to $s\in [0, t_{i,1}]$, $s\in [t_{i,1},t_{i,2}]$ and $s\ge t_{i,2}$.  As in point 3,  it is always possible to choose the sequence $t_{i,1}$ such that (\ref{eq:96}) and (\ref{eq:97}) hold in order to ensure the $L^2$ convergence of the accelerations.
\\
We need to prove that the trajectory is admissible for well chosen $t_{i,1}$ and $t_{i,2}$.\\
Let us first check that $\xi^i(s)$ remains in $\overline\Omega$ for $s\in [0,t_{i,1}]$.
Since (\ref{eq:95}) holds, we see that $(\xi^i(s)-x)\cdot n_j \le 0$.
 On the other hand,  after some algebra, we get that
\begin{displaymath}
 \begin{split}
        \eta^i(s)\cdot n_{j-1}=\frac{d\xi^i}{ds}(s)\cdot n_{j-1}= &
    v^i \cdot  n_{j-1} -\frac {s}{t_{i,1}}  \left(4- \frac {3s}{t_{i,1}} \right) \; v^i\cdot n_j \;\;  n_j\cdot n_{j-1}\\
      & -6 \frac {s}{t_{i,1}^2} \left(1- \frac {s}{t_{i,1}} \right)\; (x^i - x )\cdot n_j \;\;   n_j\cdot n_{j-1} .
    \end{split}
\end{displaymath}
Since  $ \lim_{i\to \infty} v^i\cdot n_j =0$  and 
$ \lim_{i\to \infty} v^i\cdot n_{j-1}= v\cdot n_{j-1} <0$, we see that for $i$ large enough,
\begin{displaymath}
  \lim_{i\to \infty}   \left(v^i \cdot  n_{j-1}- \frac {s}{t_{i,1}}  \left(4- \frac {3s}{t_{i,1}} \right) \; v^i\cdot n_j \;\;  n_j\cdot n_{j-1}\right)= v\cdot n_{j-1} <0,
\end{displaymath}
uniformly with respect to $s\in [0,t_{i,1}]$. On the other hand, the conditions (\ref{eq:96}) and (\ref{eq:97}) imply that $\lim_{i\to \infty} \frac  {(x^i - x )\cdot n_j}{t_i}=0$. Combining the latter two observations yields that
\begin{displaymath}
  \lim _{i\to \infty} \max_{s\in [0,t_{i,1}]} \left|  (\eta^i(s) - v)\cdot n_{j-1} \right|=0.
\end{displaymath}
Hence,  for $i$ large enough,  $\eta^i(s)\cdot n_{j-1}<0$ for all $s\in t_{i,1}$. This implies that $(\xi^i(s)-x)\cdot n_{j-1} \le 0$ for all $s\in [0,t_{i,1}]$.\\
Combining the information above and arguing essentially as in case 3.(a), we see that it is possible to choose $t_{i,1}$ satisfying (\ref{eq:96}) and (\ref{eq:97}), $t_{i,2}$ bounded away from $0$ uniformly w.r.t. $i$, such that  $(\xi^i, \eta^i)\in \Gamma[x^i, v^i]$ for $i$ large enough and  $\ds \lim_{i\to \infty} \left \|  \frac{d\eta^i}{ds}- \frac{d\eta}{ds} \right\|_{L^2 (0,T)}=0$.
\item $v^i \cdot n_j>0$ at least for $i$ large enough and (\ref{eq:88}) holds: again, the trajectory is constructed as in (\ref{eq:94}) with $t_{i,1}$ satisfying (\ref{eq:96})-(\ref{eq:97}).
As in 3.(b), a further restriction is needed on $t_{i,1}$ such that the trajectory is admissible.
\\
For $0\le s\le t_{i,1}$, $(\xi^i(s) -x)\cdot n_j$ is given by (\ref{eq:95}) and is non positive if 
\begin{equation}
  \label{eq:98}
  t_{i,1} \le    \frac {3|(x^i -x)\cdot n_j |} {v^i\cdot n_j}.
\end{equation}
On the other hand, the proof that $ (\xi^i(s) -x)\cdot n_{j-1}\le 0$ for $i$ large enough and all $s\in [0,t_{i,1}]$ is the same as in subcase 5.(a).  Hence if the sequence $(t_{i,1})$ satisfies (\ref{eq:96})-(\ref{eq:97}) and (\ref{eq:98}),  then for $i$ large enough, $\xi^ i (s)\in \overline \Omega$ for for all $s\in [0,t_{i,1}]$.
Constructing such a sequence $(t_{i,1})$ is possible thanks to (\ref{eq:88}).
\\
Then, using the fact that $v\cdot n_{j-1}<0$ and arguing as in case 3.,  it is possible to choose the sequence $t_{i,2}$ bounded from below by a positive constant independent of $i$ such that $(\xi^i,\eta^i)\in \Gamma[x^i, v^i]$ and $  \lim_{i\to \infty} \left\|\frac {d\eta^i}{dt} -\frac {d\eta}{dt} \right\| _{L^2(0,T)}=0$.

\end{enumerate}
\item Same arguments as for case 5., exchanging the roles of $j$ and $j-1$.

\item 
 The trajectory $\xi^i$ is constructed as follows:
    \begin{equation}
      \label{eq:99}
      \xi^i(s)=\left\{
        \begin{array}[c]{rcl}
          Q_{t_{i},x^i, v^i , x ,0}(s) ,\quad & \hbox{ if } &\quad s\le t_{i},\\
          \xi(s-t_{i}), \quad &\hbox{ if } &t_{i}\le t\le T,
        \end{array}
      \right.
    \end{equation}
We see that, for $k=j-1, j$ and $s\in [0,t_{i}]$, 
  \begin{displaymath}
      (\xi^i (s) -x)\cdot n_{k}= \left(1-\frac s {t_{i}}\right)^2 \left( \left( 1+2  \frac s {t_{i}}  \right) (x^i-x)\cdot n_k +s v^i\cdot n_k  \right).
    \end{displaymath}
 Hence, a sufficient condition for $\xi^i(s)$ to stay in $\overline \Omega$ for all $s\in [0,t_{i}]$ is that $ t_{i}\le 3  \min_{k\in \{j-1,j\}}  \frac {(x-x^i)\cdot n_k}{ (v^i\cdot n_k)_+}$, with the convention that $    \frac {(x-x^i)\cdot n_k }  { (v^i\cdot n_k)_+} =+\infty$ if $(v^i\cdot n_k)_+=0$.
 Then, we also need that  $\lim_{i\to \infty} t_{i}=0$ and that $\lim_{i\to \infty} \frac {  |x-x^i|^2 }{t_{i}^3} + \frac {|v^i|^2 }{t_{i}}=0$
in order to obtain that $\lim_{i\to \infty} \left \| \frac {d \eta^i} {dt}\right\| _{L^2 (0,t_i)} =0$.
From (\ref{eq:89}), it is possible to construct a sequence $(t_i)_i$ fulfilling all  the  desired properties.
\end{enumerate}
\end{proof}
Proposition~\ref{sec:contr-probl-conv-1bis} below is the counterpart of Proposition~\ref{sec:contr-probl-conv-1}:
 \begin{proposition}\label{sec:contr-probl-conv-1bis}
 Given $r>0$,  let us define $\Theta_r$ by  \eqref{eq:35}
where $K_r$ is defined by (\ref{eq:33}) and $\Theta$ is a closed subset of $\Xi^{\rm ad}$ which satisfies the assumption in Proposition \ref{sec:clos-graph-prop-14}.
\\
 Under Assumption~\ref{sec:setting-notations-0}, the value function $u$ given by \eqref{eq:7}
 is continuous  on $ \Theta_r$.\\
There exists a positive number
 $C= C(r, M)  $  such that
if  $(x,v)\in \Theta_r$, then $\Gamma^{\rm opt} [x,v]\subset  \Gamma_C$, where $\Gamma_C$ is defined in (\ref{eq:34}).
 \end{proposition}

\subsection{Mean field games  with state constraints}\label{sec:mean-field-games-2}
All the results obtained in Section \ref{sec:MFG-contr-accel} can be generalized to the case when $\Omega$ is a bounded and convex polygonal domain of $\R^2$, provided the initial distribution of states $m_0$ is supported in $\Theta_r$ defined as in Proposition \ref{sec:contr-probl-conv-1bis}.

\bigskip

\noindent {\bf Acknowledgments.}
We would like to thank P. Cardaliaguet for an enlightening discussion concerning the argument in Paragraph~\ref{sec:mean-field-games-1}.  YA and NT were  partially supported by the ANR (Agence Nationale de la Recherche) through MFG project ANR-16-CE40-0015-01. PM and CM were partially supported by GNAMPA-INdAM and by the Fondazione CaRiPaRo Project ``Nonlinear Partial Differential Equations: Asymptotic Problems and Mean-Field Games''.

{\small
\bibliographystyle{plain}
\bibliography{MFG_acc}
}

\end{document}